\documentclass[11pt]{article}

\usepackage[table, dvipsnames]{xcolor}
\usepackage[pagebackref,backref=true,colorlinks=true, linkcolor=Sepia, citecolor=Sepia]{hyperref}
\usepackage{amsmath,amsthm,amssymb,fancyhdr,graphicx,bbm,cancel,mathrsfs,todonotes,xypic,pinlabel}
\usepackage{tikz}
\usetikzlibrary{decorations.markings}
\usepackage{overpic}
\usepackage{float}
\usepackage{extarrows,tipa}
\usetikzlibrary{matrix}
\usepackage[all]{xy}
\usepackage{mathtools}
\usepackage{bm}
\usepackage{verbatim}
\usepackage{microtype}
\usepackage{subcaption}

\theoremstyle{theorem}
\newtheorem{thm}{Theorem}[section]

\newtheorem{lem}[thm]{Lemma}
\newtheorem{prop}[thm]{Proposition}

\theoremstyle{definition} 

\newtheorem{rmk}[thm]{Remark}
\theoremstyle{remark} 

\usepackage{enumerate,pdfpages,stmaryrd}
\usepackage[margin=1.4in, marginparwidth=1in]{geometry}
\usepackage{tikz}
\usepackage{dsfont}
\usetikzlibrary{matrix}

\newcommand{\C}{\mathbb{C}}

\newcommand{\Z}{\mathbb{Z}}

\newcommand{\R}{\mathbb{R}}

\newcommand{\beq}{\begin{equation*}}
\newcommand{\eeq}{\end{equation*}}

\newcommand{\SL}{\mathrm{SL}}

\newcommand{\onto}{\twoheadrightarrow}

\DeclareMathOperator{\Diff}{Diff}

\DeclareMathOperator{\PSL}{PSL}

\DeclareMathOperator{\Aut}{Aut}

\DeclareMathOperator{\Homeo}{Homeo}

\DeclareMathOperator{\SO}{SO}

\DeclareMathOperator{\Out}{Out}

\newcommand{\dmo}{\DeclareMathOperator}

\newcommand{\Om}{\Omega}

\newcommand{\what}{\widehat}\newcommand{\wtil}{\widetilde}

\newcommand{\bb}[1]{\mathbb{#1}}

\dmo{\sgn}{sign}\dmo{\Span}{span}
\dmo{\we}{\wedge}
\dmo{\ind}{ind}\dmo{\Ind}{Ind}
\dmo{\bop}{\bigoplus}\dmo{\pic}{Pic}
\dmo{\vol}{Vol}\dmo{\gal}{Gal}\dmo{\perm}{Perm}
\dmo{\tor}{Tor}\dmo{\ext}{Ext}\dmo{\Ext}{Ext}
\dmo{\aut}{aut}

\dmo{\inn}{Inn}\dmo{\var}{Var}
\dmo{\ad}{ad}\dmo{\curl}{curl}
\dmo{\hy}{\bb H}\dmo{\Sl}{SL}
\dmo{\psl}{PSL}
\dmo{\iso}{iso}
\dmo{\conf}{Conf}
\dmo{\stab}{Stab}\dmo{\Jac}{Jac }
\dmo{\diam}{diam}\dmo{\fix}{Fixed}\dmo{\Fix}{Fix}
\dmo{\injR}{injRad}\dmo{\Ad}{Ad}
\dmo{\esv}{ess-vol}
\dmo{\nil}{Nil}\dmo{\sol}{Sol}
\dmo{\Div}{div}
\dmo{\SU}{SU}
\dmo{\rk}{rk}
\dmo{\rank}{rank}
\dmo{\psp}{PSp}\dmo{\psu}{PSU}
\dmo{\PU}{PU}\dmo{\pgl}{PGL}
\dmo{\Mod}{Mod}\dmo{\range}{Range}
\dmo{\eu}{eu}\dmo{\mi}{mi}
\dmo{\Log}{Log}\dmo{\supp}{supp}
\dmo{\maps}{Maps}\dmo{\Gr}{Gr}
\dmo{\Pin}{Pin}
\dmo{\Spin}{Spin}\dmo{\Str}{Str}
\dmo{\Sq}{Sq}\dmo{\Symp}{Symp}
\dmo{\pd}{PD}\dmo{\PD}{PD}\dmo{\sig}{Sig}
\dmo{\St}{St}
\dmo{\Pt}{Pt}\dmo{\pt}{pt}

\dmo{\Pl}{PL}
\dmo{\String}{String}\dmo{\smear}{smear}

\dmo{\dev}{dev}
\dmo{\met}{Met}\dmo{\contact}{Contact}
\dmo{\teich}{Teich}\dmo{\Teich}{Teich}\dmo{\qi}{QI}
\dmo{\der}{Der}
\dmo{\cl}{Cliff}\dmo{\Cl}{Cl}
\dmo{\Pf}{Pf}
\dmo{\ch}{ch}\dmo{\diag}{diag}
\dmo{\grad}{grad}\dmo{\Char}{char}
\dmo{\spec}{Spec}\dmo{\Arg}{Arg}
\dmo{\gl}{GL}
\dmo{\sym}{Sym}\dmo{\Sym}{Sym}
\dmo{\com}{Comm}
\dmo{\Lk}{Lk}
\dmo{\CAT}{CAT}
\dmo{\Rep}{Rep}
\dmo{\Res}{Res}
\dmo{\Conf}{Conf}
\dmo{\PConf}{PConf}
\dmo{\Push}{Push}
\dmo{\Cont}{Cont}
\dmo{\sm}{\setminus}
\dmo{\vn}{\varnothing}
\dmo{\disk}{\mathbb D}
\dmo{\Trd}{Trd}\dmo{\Mat}{Mat}
\dmo{\Riem}{Riem}
\dmo{\Diffn}{\Diff_0}\dmo{\diff}{diff}
\dmo{\homeo}{Homeo}

\dmo{\Ham}{Ham}\dmo{\Met}{Met}
\dmo{\Ein}{Ein}\dmo{\CP}{\co P}
\dmo{\Per}{Per}\dmo{\Ric}{Ric}
\dmo{\Nrd}{Nrd}
\dmo{\Comp}{Comp}\dmo{\PSC}{PSC}
\dmo{\Cent}{Cent}\dmo{\Orb}{Orb}
\dmo{\aind}{a-ind}\dmo{\tind}{t-ind}
\dmo{\constant}{constant}
\dmo{\Td}{Td}
\dmo{\LMod}{LMod}
\dmo{\SMod}{SMod}
\dmo{\SDiff}{SDiff}
\dmo{\Br}{Br}
\dmo{\csch}{csch}
\dmo{\triv}{triv}
\dmo{\genus}{genus}
\dmo{\Homeq}{HomEq}
\dmo{\PP}{\mathbb{P}}
\dmo{\U}{U}
\dmo{\Gal}{Gal}
\dmo{\BDiff}{\wtil{\Diff}}
\dmo{\BAut}{\wtil{\Aut}}
\dmo{\Iso}{Iso}
\dmo{\Cone}{Cone}
\dmo{\codim}{codim}
\dmo{\II}{II}
\dmo{\I}{I}
\dmo{\InjRad}{InjRad}
\dmo{\Inn}{Inn}
\dmo{\sys}{sys}
\dmo{\Comm}{Comm}
\dmo{\PO}{PO}
\dmo{\vertex}{Vert}
\dmo{\POm}{P\Om}
\dmo{\ab}{ab}
\dmo{\PSO}{PSO}
\dmo{\CRS}{CRS}
\dmo{\Diffext}{Diffext}
\dmo{\Diffextad}{Diffextad}
\dmo{\Diffstand}{Diffstand}

\usepackage[titles]{tocloft}

\begin{document}

\title{Sphere decomposition for finite subgroups of 3-manifold mapping class groups}
\author{Corey Bregman and Bena Tshishiku}

\maketitle


\begin{abstract}
For a closed, oriented 3-manifold $M$, we show that any finite subgroup of its mapping class group preserves an isotopy class of filling sphere system on $M$. This is an isotopy version of the ``equivariant sphere theorem" of Meeks--Yau. As part of our analysis, we establish a bijection between simplices of the Guirardel--Levitt outer space for $\pi_1(M)$ and filling sphere systems in $M$, generalizing an argument of Hatcher. We also compute the image of $\Homeo(M)\to\Out(\pi_1M)$ for arbitrary $M$. 
\end{abstract}

\setcounter{secnumdepth}{2}
\setcounter{tocdepth}{1}
\begingroup
\linespread{0.8}\selectfont 
\tableofcontents
\addtocontents{toc}{\protect\thispagestyle{empty}}
\endgroup

\section{Introduction}

Let $M$ be a closed, oriented 3-manifold. In this paper we prove a structure theorem for finite subgroups of the mapping class group $\Mod(M)$ when $M$ is reducible. The following theorem of Meeks--Yau \cite{meeks-yau} (see also \cite{dunwoody,plotnick-sphere-theroem}) motivates our main result Theorem \ref{thm:mod-sphere-decomposition}. 

\begin{thm}[Meeks-Yau, sphere decomposition theorem]\label{thm:meeks-yau}
Let $M$ be a closed oriented reducible $3$-manifold. Given a finite subgroup $G<\Diff(M)$, there exists an essential, filling sphere system $S\subset M$ that is preserved by $G$. 
\end{thm}

Here a sphere system $S\subset M$ is an embedded submanifold homeomorphic to a disjoint union of 2-spheres; we say that  $S$ is \emph{essential} if no component of $S$ bounds a ball in $M$; a sphere system $S$ is \emph{filling} if the manifold obtained from $M$ by deleting a regular neighborhood of $S$ and attaching a 3-ball to each boundary component is irreducible.

Our main result is an analogue of Theorem \ref{thm:meeks-yau}, where $\Diff(M)$ is replaced by the mapping class group $\Mod(M)$, the group of diffeomorphisms of $M$ modulo isotopy. 

\begin{thm}\label{thm:mod-sphere-decomposition}
Let $M$ be a closed, oriented reducible $3$-manifold. Given a finite subgroup $G<\Mod(M)$, there exists an isotopy class of essential, filling sphere system $S\subset M$ that is preserved by $G$. 
\end{thm}

Note that Theorem \ref{thm:mod-sphere-decomposition} does not follow directly from Theorem \ref{thm:meeks-yau}. For such an implication, one would need to know additionally that $G<\Mod(M)$ lifts to $\Diff(M)$, i.e.\ a solution to the Nielsen realization problem. This is not generally true; moreover, our motivation for proving Theorem \ref{thm:mod-sphere-decomposition} is to use it as a tool for studying the Nielsen realization problem (not the other way around).

\subsection{A motivation (Nielsen realization)}

A form of the Nielsen realization problem asks when a finite subgroup $G<\Out(\pi_1M)$ lifts to $\Diff(M)$ under the natural map $\Diff(M)\to\Out(\pi_1M)$. In other words, when can a group of symmetries of $\pi_1(M)$ be promoted to a group of symmetries of $M$. This problem has a long history and connects to geometry (e.g.\ Teichm\"uller theory), the Borel conjecture for aspherical manifolds, and flat connections on fiber bundles. 

Since the homomorphism $\Diff(M)\to\Out(\pi_1M)$ factors through the mapping class group $\Mod(M)$, it is reasonable to consider the intermediate problems of lifting $G<\Out(\pi_1M)$ to $\Mod(M)$ and lifting $G<\Mod(M)$ to $\Diff(M)$. For the latter (i.e.\ $\Diff\to\Mod$ lifting), Theorems \ref{thm:meeks-yau} and \ref{thm:mod-sphere-decomposition} give an approach to this problem when $M$ is a reducible 3-manifold, and allow one to reduce to realization problems for irreducible manifolds. Furthermore, Theorem \ref{thm:mod-sphere-decomposition} can be similarly combined with a result of Hensel--Kielak to study the former lifting problem (i.e.\ $\Mod\to\Out$ lifting). We plan to elaborate on this in future work.

In the context of these lifting problems, there is another intermediate. Let $\Diff(M)^\delta$ denote $\Diff(M)$ but with the discrete topology. We can factor the induced map on classifying spaces $B\Diff(M)^{\delta}\to B\Mod(M)$ through $B\Diff(M)$. These lifting problems can be interpreted topologically in terms of bundle monodromy and flat connections. In analogy to Theorem \ref{thm:mod-sphere-decomposition}, it would be interesting to know, for every bundle $M\to E\to BG$ with structure group $\Diff(M)$ and $G$ finite, if there is an essential, filling sphere system $S\subset M$, so that the structure group reduces to the group $\Diff(M,S)$ of diffeomorphisms that preserve $S$.

\subsection{\boldmath Case study: $M=(S^1\times S^2)^{\#n}$}

We sketch a proof of Theorem \ref{thm:mod-sphere-decomposition} in the special case $M\cong (S^1\times S^2)^{\#n}$. This will allow us to explain what was previously known and also indicate some subtleties in the general case. 

Fix $n\ge1$, set $M= (S^1\times S^2)^{\#n}$, and let $G<\Mod(M)$ be a finite subgroup. We want to find an essential, filling sphere system $S\subset M$ whose isotopy class is invariant under $G$.  

Let $G'$ denote the image of $G$ under the homomorphism $\Mod(M)\to\Out(F_n)$ given by the action on $\pi_1(M)\cong F_n$ (free group). The finite group $G'<\Out(F_n)$ is realized by  automorphisms of a simplicial graph $X$ with $\pi_1(X)\cong F_n$ by Culler, Khramtsov, and Zimmermann \cite{culler,khramtsov,zimmermann} (building on work of Karrass--Pietrowski--Solitar). Minimal marked graphs with $\pi_1(X)\cong F_n$ are in bijection with simplices of Culler--Vogtmann outer space \cite{culler-vogtmann}, and by an argument of Hatcher \cite{hatcher-stability-autfn}, the marked graph $X$ determines a unique isotopy class of essential, filling sphere system $S\subset M$. By construction, $S$ is invariant under $G$ (here it helps to note that the kernel of $\Mod^+(M)\to\Out(F_n)$ is generated by sphere twists, which act trivially on every isotopy class of sphere system). This proves Theorem \ref{thm:mod-sphere-decomposition} for $M=(S^1\times S^2)^{\#n}$. 

For arbitrary $M$, with the exception of the argument of Hatcher, the ingredients of this proof sketch have already been appropriately generalized. Hensel--Kielak \cite{hensel-kielak} give a suitable generalization of the result of Karrass--Pietrowski--Solitar. The Culler--Vogtmann outer space (which is really only an intermediate in our argument that we nevertheless find useful for framing parts of the discussion) can be replaced by the more general outer space of a free product defined by Guirardel--Levitt \cite{guirardel-levitt}. The technical difficulty in generalizing Hatcher's argument relates to the fact that $\Homeo(M)\to\Out(\pi_1M)$ is surjective for $M=(S^1\times S^2)^{\# n}$ \cite{laudenbach} but not for general 3-manifolds. Ultimately, we find a different argument that gives the same conclusion, establishing a bijection between simplices of Guirardel--Levitt outer space for $\pi_1(M)$ and essential filling sphere systems on $M$. See Theorem \ref{thm:sphere-tree}. 

\subsection{\boldmath Image of $\Homeo(M)\to\Out(\pi_1M)$}

As indicated in the preceding section, an important ingredient in our proof of Theorem \ref{thm:mod-sphere-decomposition} is to understand the failure of $\Homeo(M)\to\Out(\pi_1M)$ to be surjective. There are three basic examples where this occurs. 
\begin{enumerate}[(i)]
\item $M=N\#M'$ where $N$ is irreducible and $\Homeo(N)\to\Out(\pi_1N)$ is not surjective. 
\item $M=N_1\#N_2\#M'$, where $\pi_1(N_1)$ and $\pi_1(N_2)$ are isomorphic but $N_1$ and $N_2$ are not homeomorphic. 
\item $M=N_1\#N_2\#M'$, where $N_1,N_2$ are irreducible and there exists $\phi_i\in\Out(\pi_1N_i)$ realized by a homeomorphism $f_i$ of $N_i$, but $f_1,f_2$ cannot be chosen with the same orientation sign. Similar examples exist where the factors are permuted, but not with consistent orientation. 
\end{enumerate} 


We remark that (i) happens only when $N$ is spherical and (ii) happens only when $N_1,N_2$ are lens spaces, but (iii) is quite common. For example, it will occur whenever $M=P\#Q$ where $P,Q$ both have orientation-reversing homeomorphisms.

We show (i)--(iii) are essentially the only phenomena responsible for $\Homeo(M)\to\Out(\pi_1M)$ not being surjective and give a  description of the image. See Proposition \ref{prop:image-mod-out}. In particular, we find that the image of $\Homeo(M)\to\Out(\pi_1M)$ depends only on knowing which irreducible factors are homeomorphic (remembering orientation) and the image of $\Homeo(N)\to\Out(\pi_1N)$ for each irreducible factor. We give a complete account of the image of $\Homeo(N)\to\Out(\pi_1N)$ for irreducible $N$, which is summarized in the following theorem. 

\begin{thm}[Automorphisms induced by homeomorphisms]\label{thm:image}
Let $N$ be a closed, orientable, irreducible $3$-manifold. If $N$ is not covered by $S^3$, then $\Homeo(N)\to\Out(\pi_1N)$ is surjective. If $N$ is spherical, then $\phi\in\Out(\pi_1N)$ is induced by a homeomorphism of $N$ if and only if it has the following two properties. 
\begin{enumerate}[(i)]
\item The automorphism $\phi$ acts on $H_3(\pi_1N)$ by multiplication by $\pm1$. 
\item For any cyclic subgroup $C<\pi_1(N)$ preserved by $\phi$, the induced automorphism of $C$ preserves the Reidemeister torsion of the corresponding lens space cover $L\to N$.  
\end{enumerate} 
\end{thm}

The case when $N$ is not covered by $S^3$ follows directly from several previous works (see Section \ref{sec:image-irreducible}). Our main contribution to the statement is the spherical case. 

Condition (i) in Theorem \ref{thm:image} is equivalent to $\phi$ preserving the third $k$-invariant of $N$; see Remark \ref{rmk:k-invariants}. Roughly Theorem \ref{thm:image} says that for spherical $N$, the Reidemeister torsion and $k$-invariants give complete obstructions for $\phi\in\Out(\pi_1N)$ to be induced by a homeomorphism of $N$.

\subsection{Section outline}

Sections \ref{sec:outer-space} and \ref{sec:hatcher-argument} contain background results. In Section \ref{sec:outer-space} we recall the definition of the outer space $\mathcal{PO}(\Gamma)$ of a group $\Gamma$ with a nontrivial free product decomposition, following Guirardel--Levitt, and we state the result of Hensel--Kielak that gives fixed points for finite subgroups $G<\Out(\Gamma)$ acting on $\mathcal{PO}(\Gamma)$. In Section \ref{sec:hatcher-argument} we recall the argument of Hatcher relating simplices of the outer space for $F_n$ to sphere systems on $(S^1\times S^2)^{\#n}$; in the process we will see how the image of $\Homeo(M)\to\Out(\pi_1M)$ is relevant for generalizing Hatcher's argument. 

Sections \ref{sec:image} and \ref{sec:hatcher-generalized} contain technical ingredients for the proof of Theorem \ref{thm:mod-sphere-decomposition}. In Section  \ref{sec:image} we describe the image of $\Homeo(M)\to\Out(\pi_1M)$, and in Section \ref{sec:hatcher-generalized} we generalize Hatcher's argument to arbitrary $M$. From these ingredients, the proof of Theorem \ref{thm:mod-sphere-decomposition} follows quickly and it is also given in Section \ref{sec:hatcher-generalized}. 

Finally, in Section \ref{sec:image-irreducible}, we compute the image of $\Homeo(N)\to\Out(\pi_1N)$ for irreducible $N$, focusing primarily on spherical $N$, since other cases were previously known. 

\paragraph{Computational resource disclosure.} The authors used LLMs to check the proofs and proofread the paper. The authors take full responsibility for the content of this article. 

\section{Outer space for a free product of groups}\label{sec:outer-space}

Guirardel--Levitt defined an outer space $\mathcal{PO}(\Gamma)$ for a free product of groups $\Gamma$. We recall how the set $\mathcal X(\Gamma)$ of simplices of $\mathcal{PO}(\Gamma)$ is defined and state a result of Hensel--Kielak which gives fixed points for finite subgroups of $\Aut(\Gamma)$ acting on $\mathcal X(\Gamma)$.

Fix a finitely-generated group $\Gamma$, decomposed as a free product $\Gamma=\Gamma_1*\cdots\Gamma_m*F_n$, where each $\Gamma_i$ is indecomposable and not infinite cyclic. Such a decomposition exists and the conjugacy classes of the $\Gamma_i$'s are unique by Grushko's theorem. 

We call a simplicial $\Gamma$-tree a \emph{Grushko tree for $\Gamma$} if it has the following properties. 
\begin{enumerate}
\item[(i)] Minimal: there is no proper $\Gamma$-invariant subtree, and each degree-2 vertex $v$ is fixed by an element of $G$ that exchanges the edges incident to $v$. 
\item[(ii)] Stabilizers: for each $1\le i\le m$, there is exactly one orbit of vertices whose stabilizer is conjugate to $\Gamma_i<\Gamma$, and all other point-stabilizers are trivial.
\end{enumerate} 

\begin{rmk}
In the (degenerate) case $\Gamma=\Z$ (in which case $M=S^2\times S^1$), we should drop the condition on degree-2 vertices, so that $\mathcal X(\mathbb Z)$ is a singleton, rather than empty.
\end{rmk}

Each Grushko tree $T$ gives a Grushko decomposition of $\Gamma$. Indeed by Bass--Serre theory, the quotient $X=T/\Gamma$ has the structure of a graph of groups $(X,\mathcal G)$ whose fundamental group $\pi_1(X,\mathcal G)$ is presented as a free product $\Gamma_1*\cdots*\Gamma_m*F_n$ and is ``marked" by an isomorphism $\Gamma\cong\pi_1(X,\mathcal G)$, which is well-defined up to inner automorphisms.

Define $\mathcal X(\Gamma)$ as the set of Grushko trees for $\Gamma$, up to $\Gamma$-equivariant isomorphism. The set $\mathcal X(\Gamma)$ indexes simplices of the outer space $\mathcal {PO}(\Gamma)$ defined by Guirardel--Levitt \cite{guirardel-levitt}, whose points are \emph{metric} Grushko trees, up to rescaling all edge lengths by a positive number.

There is a natural action of $\Out(\Gamma)$ on $\mathcal X(\Gamma)$. Viewing a point of $\mathcal X(\Gamma)$ as a homomorphism $\alpha:\Gamma\to\Aut(T)$, the group $\Aut(\Gamma)$ acts by precomposition, and inner automorphisms act trivially, since $\alpha\circ \mathrm{conj}(\gamma)$ is conjugate to $\alpha$ by $\alpha(\gamma)\in\Aut(T)$.

The following result is proved by Hensel--Kielak \cite[Thm.\ 6.1]{hensel-kielak}. 

\begin{thm}[Hensel--Kielak]\label{thm:hensel-kielak}
Let $\Gamma$ be a finitely generated group, and let $G<\Out(\Gamma)$ be a finite subgroup. Then the action of $G$ on $\mathcal{PO}(\Gamma)$ fixes a barycenter. Equivalently, the action of $G$ on $\mathcal X(\Gamma)$ has a fixed point. 
\end{thm}

The case $\Gamma=F_n$ of Theorem \ref{thm:hensel-kielak} was proved earlier by Culler, Khramtsov, and Zimmermann (independently). In this case $\mathcal{PO}(F_n)$ is the same as the Culler--Vogtmann outer space. We also note Theorem \ref{thm:hensel-kielak} is most interesting when $\Gamma$ is decomposable, since otherwise $\mathcal X(\Gamma)$ is a point. We record the following lemma for future reference.

\begin{lem}\label{lem:orbits}
The orbits of $\Out(\Gamma)$ acting on $\mathcal X(\Gamma)$ are in bijection with unmarked graph of groups decompositions for $\Gamma$, where the edge groups are trivial and the vertex groups are indecomposable. 
\end{lem}

The lemma follows directly from Bass--Serre theory, which gives that a point of $\mathcal X(\Gamma)$ can equivalently be viewed as a marked graph of groups (with the properties stated in the lemma), where the marking is an isomorphism of the fundamental group of the graph of groups with $\Gamma$. The action of $\Aut(\Gamma)$ changes the marking. 

\section{Recollection of Hatcher's argument}\label{sec:hatcher-argument}

In this section we recall Hatcher's argument that, for the free group $F_n$, simplices of $\mathcal{PO}(F_n)$ are in bijection with isotopy classes of essential, filling sphere systems on $(S^1\times S^2)^{\#n}$. This will be generalized in Theorem \ref{thm:sphere-tree}. 

We begin with a general setup. Fix a closed orientable 3-manifold 
\[M=M_1\#\cdots\# M_m\#(S^1\times S^2)^{\#n}.\] Write $\mathcal S(M)$ for the set of isotopy classes of essential, filling sphere systems in $M$, where we additionally require that no two spheres in the sphere system are parallel (bound an $S^2\times[0,1]$). 
The pre-image of a sphere system $S\subset M$ in the universal cover has a dual tree with an action of $\pi_1(M)$; this defines a map
\[\Phi:\mathcal S(M)\to\mathcal X(\pi_1M),\]
where $\mathcal X(\pi_1M)$ is the set of Grushko trees for $\pi_1(M)$, c.f.\ Section \ref{sec:outer-space}. 
Ultimately, we show that $\Phi$ is a bijection. In this section, we explain how this works for $M=(S^1\times S^2)^{\#n}$ with $n\ge2$, following an argument of Hatcher. In Section \ref{sec:hatcher-generalized} we generalize this argument to arbitrary $M$. 

For the rest of this section, we fix $n\ge2$ and set $M=(S^1\times S^2)^{\#n}$. In this case $\pi_1(M)$ is a free group $F_n$. 

It will be convenient to identify $\mathcal X(F_n)$ as the set of equivalence classes of minimal, marked graphs with rank $n$ (\`a la Culler--Vogtmann). 
Here minimal means the graph has no vertices of degree 1 or 2. The marking is a homotopy equivalence $f:X_0\to X$, where $X_0=\vee_n S^1$ (our fixed choice of reference graph). Two markings $(X,f)$ and $(X',f')$ are equivalent if there is a homeomorphism $h:X\to X'$ so that $hf$ and $f'$ are homotopic. An element $T\in\mathcal X(F_n)$ determines a rank-$n$ graph $X=T/F_n$ and a marking can be obtained by fixing a basepoint $*\in T$, identifying $F_n=\pi_1(X_0)$ in the standard way, and sending the $i$-th loop in $X_0=\vee_n S^1$ to the image of the path in $T$ from $*$ to $\gamma_i(*)$. 

With the identification of $\mathcal X(F_n)$ with marked graphs, the map $\Phi:\mathcal S(M)\to\mathcal X(F_n)$ takes the following form. 
View $M$ as obtained from $S^3$ by attaching $n$ handles. Then the union of belt spheres in each handle defines a sphere system $S_0$. The dual graph $X(S_0)$ to this sphere system is $X_0$ (maybe after a change of coordinates). The dual graph $X(S)$ of an arbitrary sphere system $S\in\mathcal S(M)$ has a marking $f_S:X(S_0)\to X(S)$ obtained by a composition $X(S_0)\to M\to X(S)$, where the first map is an inclusion of the dual graph and the second map is a projection that collapses complementary components of (a regular neighborhood of) $S$ to points (and collapses each component of the regular neighborhood to an interval). 

The following result appears in \cite[Appendix]{hatcher-stability-autfn}.

\begin{prop}[Hatcher]\label{prop:hatcher}
The map $\Phi:\mathcal S(M)\to\mathcal X(F_n)$ defined by $\Phi(S)= (X(S),f_S)$ is a bijection.
\end{prop}

\begin{proof}
Define an inverse $\Psi:\mathcal X(F_n)\to\mathcal S(M)$. Given $f:X_0\to X$, build a manifold $M_X\cong M$ by taking a copy of $S^3$ with the interior of $\deg(v)$ disks removed for each vertex $v$, and gluing these according to the edge relations in $X$. The spheres along which the gluing occurs define a filling sphere system $S_X\subset M_X$. Using that $\Homeo(M)\to\Out(\pi_1M)$ is surjective by \cite{laudenbach}, one shows that there is a homeomorphism $F:M\to M_X$ that makes the following diagram commute on the level of $\pi_1(-)$.\footnote{Fix any homeomorphism $h: M\to M_X$. The composition $M\onto X_0\xrightarrow{f} X\hookrightarrow M_X\xrightarrow{h^{-1}} M$ defines an (outer) automorphism $\alpha:\pi_1(M)\to\pi_1(M)$ (note that vertical maps in the diagram are $\pi_1$-isomorphisms). Now choose $a\in\Homeo(M)$ inducing $\alpha$, and define $F=h\circ a$.}
\[\begin{xy}
(-20,0)*+{M}="A";
(0,0)*+{M_X}="B";
(-20,-15)*+{X_0}="C";
(0,-15)*+{X}="D";
{\ar@{-->}"A";"B"}?*!/_3mm/{F};
{\ar"C";"D"}?*!/_3mm/{f};
{\ar@{->>} "A";"C"}?*!/_3mm/{};
{\ar@{^{(}->} "D";"B"}?*!/^3mm/{};
\end{xy}\]
Here the vertical maps are the projection and inclusion associated to the sphere systems $S_0\subset M$ and $S_X\subset M_X$, respectively. Since $\Homeo^+(M)\to\Out(\pi_1M)$ is surjective we can furthermore choose $F$ to be orientation-preserving. Given this, define $\Psi(X,f)=F^{-1}(S_X)$. This sphere system of $M$ is well-defined up to isotopy because $F\in\Homeo^+(M)$ is well-defined up to sphere twists (along the $S_X$), which act trivially on $S_X$.

To check that $\Phi$ and $\Psi$ are inverses, first fix $S\in\mathcal S(M)$. To compute $\Psi(\Phi(S))=\Psi(X(S),f_S)$, observe that the pair $(M_{X(S)},S_{X(S)})$ in the construction of $\Psi(X(S),f_S)$ is canonically homeomorphic to $(M,S)$, so we can take the marking $F:M\to M_{X(S)}=M$ to be the identity, in which case $\Psi(X(S),f_S)=S$. 

For the reverse composition, fix $(X,f)\in\mathcal X(F_n)$. By definition, 
\[\Phi(\Psi(X,f))=\Phi(F^{-1}(S_X))=(X(F^{-1}(S_X)), f_{F^{-1}(S_X)}),\]
and $X(F^{-1}(S_X))$ is the dual graph of $F^{-1}(S_X)\subset M$, or equivalently, the dual graph of $S_X\subset M_X$. Consequently, the marking $f_{F^{-1}(S_X)}$ is the map $X_0\hookrightarrow M\xrightarrow{F} M_X\onto X$, which is homotopic to $f:X_0\to X$ by construction. This completes the proof. 
\end{proof}

\section{\boldmath Image of $\Homeo(M)\to\Out(\pi_1M)$ for reducible $M$} \label{sec:image}

In this section, for a closed oriented reducible $3$-manifold $M$, we describe the image of $\Homeo(M)\to\Out(\pi_1M)$ using work of Fouxe-Rabinovitch \cite{FR1,FR2}. See Proposition \ref{prop:image-mod-out} for a precise statement. This should be viewed as a relative computation, since part of the answer is expressed in terms of the image of $\Homeo(N)\to\Out(\pi_1N)$ for $N$ an irreducible factor of $M$. This is complemented by a computation of the image of $\Homeo(N)\to\Out(\pi_1N)$ for arbitrary irreducible $N$ given in Section \ref{sec:image-irreducible}. 

Fix $M=M_1\#\cdots\# M_m\#(S^1\times S^2)^{\#n}$ and write $\Gamma=\Gamma_1*\cdots*\Gamma_m*F_n$ 
for the corresponding decomposition of the fundamental group $\Gamma:=\pi_1(M)$. In particular $\Gamma_i\cong\pi_1(M_i)$ for $i=1,\ldots,m$. 

In Proposition \ref{prop:image-mod-out}, we describe the image of $\Homeo(M)\to\Out(\Gamma)$ in terms of a surjective homomorphism
\begin{equation}\label{eqn:factor-auts}q:\Out(\Gamma)\to\prod\Out(\Gamma_i)\rtimes P_\Gamma,\end{equation}
where $P_\Gamma$ is the group of permutations of $\{1,\ldots,m\}$ that only permute indices of isomorphic factors $\Gamma_i\cong\Gamma_j$. To give the precise statement, we use the following notation. Define $P_M<P_\Gamma$ as the subgroup of permutations of $\{1,\ldots,m\}$ that only permute indices of factors $M_i\cong M_j$ that are homeomorphic (ignoring orientation). Let $\Out_{M_i}(\Gamma_i)$ denote the image of $\Homeo(M_i)\to\Out(\Gamma_i)$. For each $M_i$, the degree of a homeomorphism realizing an automorphism defines a homomorphism 
\[
\Out_{M_i}(\Gamma_i)\to \{\pm1\}/A_i,
\]
where $A_i<\{\pm1\}$ is the subgroup of possible degrees of a homeomorphism of $M_i$ that induces the identity on $\Gamma_i$. The group $A_i$ is trivial unless $M_i=\mathbb RP^3$ in which case $A_i=\{\pm1\}$. Using this, define a homomorphism 
\begin{equation}\label{eqn:orientation-sign}
r:\prod_{i=1}^m\Out_{M_i}(\Gamma_i)\rtimes P_M\to \prod_iC_2/A_i\rtimes P_M
\end{equation}
by realizing an element of the domain as a homeomorphism $f$ of $\coprod_{i=1}^m M_i$ and defining 
\[r(f)=(\epsilon_1(f),\ldots,\epsilon_m(f),\sigma_f),\]
where $\sigma_f\in P_M$ is the permutation of components induced by $f$, and $\epsilon_i(f)$ is the degree of the restriction of $f$ to $M_i\to M_{\sigma_f(i)}$ with respect to the fixed orientations on the $M_1,\ldots,M_m$.

\begin{prop}\label{prop:image-mod-out}
Fix $M=M_1\#\cdots\#M_m\#(S^1\times S^2)^{\#n}$ and $\Gamma=\Gamma_1*\cdots*\Gamma_m*F_n$ as above. Let $q,r$ be the homomorphisms defined in $(\ref{eqn:factor-auts})$ and $(\ref{eqn:orientation-sign})$. Then the image of $\Homeo(M)\to\Out(\Gamma)$ is the subgroup $(r\circ q)^{-1}(\Delta\rtimes P_M)$, where $\Delta<\prod C_2/A_i$ is the diagonal subgroup (ignoring trivial factors of the product). 
\end{prop}

\begin{rmk}
Hatcher--Wahl \cite[Prop.\ 2.2]{hatcher-wahl} give a sufficient condition for the homomorphism $\Homeo^+(M)\to\Out(\Gamma)$ to be surjective. Proposition \ref{prop:image-mod-out} strengthens their result, essentially proving a converse. Note additionally that we do not restrict to orientation-preserving homeomorphisms.  
\end{rmk}

\subsection{\boldmath Structure of $\Aut(\Gamma)$}

To understand the image of $\Homeo(M)\to\Out(\Gamma)$, we first recall the algebraic structure of $\Aut(\Gamma)$. By work of Fouxe-Rabinovitch \cite{FR1,FR2} (see also \cite{gilbert}), $\Aut(\Gamma)$ is a product of two subgroups: 
\[
\Aut(\Gamma)=W\left[\prod\Aut(\Gamma_i)\rtimes P_\Gamma\right]. 
\]
The subgroup $\prod\Aut(\Gamma_i)\rtimes P_\Gamma$ is the group of factor automorphisms and permutations. The group $P_\Gamma$ is isomorphic to the group of permutations of $\{1,\ldots,m\}$ that only permute indices of isomorphic factors $\Gamma_i\cong\Gamma_j$. The subgroup $W$ is normal in $\Aut(\Gamma)$ and is generated by Whitehead automorphisms and factor inversions. To define these, fix a free basis $s_1,\ldots,s_n$ for the free factor $F_n$ of $\Gamma$. Consider the following types of automorphisms (here we specify an automorphism by its action on the $G_i$ and the $s_j$, and we only specify the action on the factor(s) where the automorphism is nontrivial). 
\begin{enumerate}[(i)]
\item Conjugate $G_k$ by either $a_\ell\in G_\ell$ for $\ell \neq k$ or by $s_j$. 
\item Multiply $s_k$ by either $s_\ell$ for $\ell \neq k$ or by $a_i\in G_i$, either on the left or the right. 
\item Send $s_k$ to $s_k^{-1}$. 
\end{enumerate} 
Types (i) and (ii) are called Whitehead automorphisms and automorphisms of type (iii) are called factor inversions. 

Although $W$ is normal, there is typically not a well-defined projection $\Aut(\Gamma)\to\prod\Aut(\Gamma_i)\rtimes P_\Gamma$ because the two subgroups $W$ and $\prod\Aut(\Gamma_i)\rtimes P_\Gamma$ have nontrivial intersection when $n\ge1$ (recall that $n$ denotes the rank of the largest free factor in $\Gamma$). For example if $\Gamma=A*\langle s\rangle$, then if $a\in A$ is fixed and $c,\ell,r\in\Aut(\Gamma)$ are defined by $c(x)=sxs^{-1}$ for $x\in A$ and $\ell(s)=as$ and $r(s)=sa$ (all unspecified values of $c,\ell,r$ on generators of $\Gamma=\langle A,s\rangle$ are the identity), then $c\ell c^{-1} r^{-1}$ is the factor automorphism given by conjugation by $a^{-1}$ on $A$.  Nevertheless, there is a homomorphism  
\[
q:\Out(\Gamma)\to\prod\Out(\Gamma_i)\rtimes P_\Gamma. 
\]
The kernel of $q$ is isomorphic to $W/\mathrm{Inn}(\Gamma)$. 
See \cite[\S12.3]{canary-mccullough} for further discussion. There this homomorphism is constructed using the Fouxe-Rabinovitch presentation for $\Aut(\Gamma)$.

\subsection{Proof of Proposition \ref{prop:image-mod-out}}

Our proof of Proposition \ref{prop:image-mod-out} involves some computation, so it will help to fix a model for $M$. We construct $M$ as in Figure \ref{fig:prime-decomposition} with the prime summands each glued to a single 3-sphere, and we take $S$ to be the associated sphere system. Fixing a basepoint $*\in M$, we have an obvious free product decomposition $\Gamma=\Gamma_1*\cdots*\Gamma_m*\Z^{*n}$ of $\Gamma=\pi_1(M,*)$. 

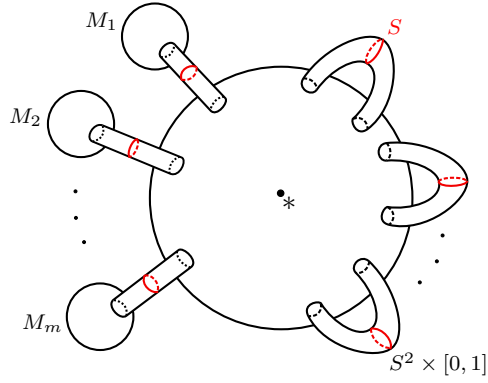
\begin{figure}[h]
\centering
\begin{overpic}[scale=1]{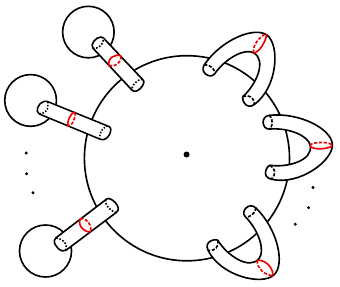}
\put (53,35){ $*$}
\put (10,77){\scriptsize$M_1$}
\put (-7.5,55){\scriptsize$M_2$}
\put (-4.5,7.5){\scriptsize$M_m$}
\put (79.5,75.5){\scriptsize\textcolor{red}{$S$}}
\put (80,-2){\scriptsize$S^2\times[0,1]$}
\end{overpic}
\caption{ A schematic for the sphere system $S$, pictured in red. Here $M$ is obtained from the disjoint union of $M_1,\ldots, M_m$ and a single copy of $S^3$ by attaching 1-handles $S^2\times [0,1]$. The sphere system $S$ is the collection of belt spheres.   }
\label{fig:prime-decomposition}
\end{figure}

We prove the proposition by showing both containments. Set 
\[\mathcal I:=(r\circ q)^{-1}(\Delta\rtimes P_M).\]
First we show the image of $\Homeo(M)\to\Out(\Gamma)$ is contained in $\mathcal I$. Suppose that $\phi\in\Out(\Gamma)$ is induced by a homeomorphism $f:M\to M$. According to \cite[Lem.\ 1.4]{mccullough}, there is a homeomorphism $s:M\to M$ that is a composition of slides so that $sf(S)=S$ for the standard sphere system $S$ from Figure \ref{fig:prime-decomposition}. Slide homeomorphisms induce Whitehead automorphisms on $\pi_1(M)$, so $sf$ and $f$ induce the same automorphism $\phi$ on $\pi_1(M)$. Since $sf$ preserves $S$, the image of $\phi$ in $\prod\Out(\Gamma_i)\rtimes P_\Gamma$ is clearly contained in $\prod\Out_{M_i}(\Gamma_i)\rtimes P_\Gamma$, and since $sf$ is a homeomorphism (with degree either $\pm1$), the further image of $sf$ in $\prod C_2/A_i\rtimes P_M$ is contained in $\Delta\rtimes P_M$. 

For the converse, suppose that $\phi$ belongs to $\mathcal I\subset \Out(\Gamma)$. To show $\phi$ is induced by a homeomorphism, first take a homeomorphism $g_1$ of $\coprod M_i$ that induces $q(\phi)\in\prod\Out_{M_i}(\Gamma_i)\rtimes P_M$. Using $g_1$ and our model for $M$ (Figure \ref{fig:prime-decomposition}), we can build a homeomorphism $g$ of $M$ such that $q(g_*)=q(\phi)$. Then $\phi\circ g_*^{-1}$ belongs to the kernel of $q$, and it suffices to show that any element of $\ker(q)<\Out(\Gamma)$ is induced by a homeomorphism of $M$. This is well-known and is explained, for example, in \cite[Proof of Prop.\ 2.2]{hatcher-wahl}. We illustrate it with the figures below. This completes the proof of Proposition \ref{prop:image-mod-out}. 
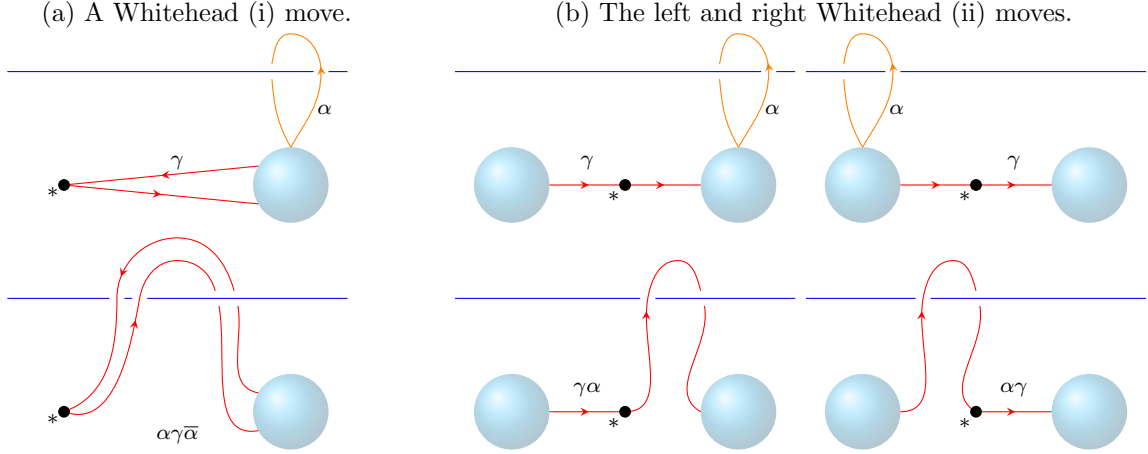
\begin{figure}[H]

\begin{subfigure}[t]{.35\textwidth}
\subcaption{A Whitehead (i) move.}
\begin{tikzpicture}
\tikzset{
  mid arrow/.style={
    postaction={
      decorate,
      decoration={
        markings,
        mark=at position 0.5 with {\arrow[#1]{stealth}}
      }
    }
  },
  mid arrow/.default=black
}
\shade[ball color=cyan!70!blue, opacity=.3] (3,3) circle (.5cm);
\draw[blue] (-.75,4.5)--(3.3,4.5);
\draw[blue](3.5,4.5)--(3.75,4.5);
\draw[red, mid arrow=red](2.57,3.25)--(0,3);
\draw[red, mid arrow=red](0,3)--(2.57,2.75);
\draw[color=orange,mid arrow=orange] (3,3.5) to[out=60,in=270](3.4,4.5) to[out=90,in=0](3,5) to[out=180,in=80](2.75,4.6);
\draw[color=orange](2.75, 4.4) to[out=275,in=120](3,3.5);
\filldraw (0,3) circle (2pt);
\node at (-.15,2.9){\scriptsize$*$};
\node at (3.45,4){\scriptsize$\alpha$};
\node at (1.5,3.3){\scriptsize$\gamma$};

\shade[ball color=cyan!70!blue, opacity=.3] (3,0) circle (.5cm);
\draw[blue] (-.75,1.5)--(.6,1.5);
\draw[blue] (.8,1.5) -- (.9,1.5);
\draw[blue](1.1,1.5)--(3.75,1.5);
\draw[red] (2.57,.25) to[out=170,in=280](2.3,1.4);
\draw[red,mid arrow=red] (2.25,1.6) to[out=105,in=0](1.5,2.3) to[out=180,in=90](.7,1.5) to[out=270,in=20] (0,0);
\draw[red, mid arrow=red] (0,0) to[out=-20,in=260] (1,1.5) to[out=80,in=180](1.5,2) to[out=0,in=105](2,1.6);
\draw[red] (2.05,1.4) to[out=280,in=190](2.57,-.25);
\filldraw (0,0) circle (2pt);
\node at (-.15,-.1){\scriptsize$*$};
\node at (1.5, -.3){\scriptsize$\alpha\gamma\overline{\alpha}$};
\end{tikzpicture}
\end{subfigure}
\qquad
\begin{subfigure}[t]{.66\textwidth}
%
\caption{The left and right Whitehead (ii) moves.}
\begin{tikzpicture}
\tikzset{
  mid arrow/.style={
    postaction={
      decorate,
      decoration={
        markings,
        mark=at position 0.52 with {\arrow[#1]{stealth}}
      }
    }
  },
  mid arrow/.default=black
}
\shade[ball color=cyan!70!blue, opacity=.3] (0,3) circle (.5cm);
\shade[ball color=cyan!70!blue, opacity=.3] (3,3) circle (.5cm);
\draw[blue] (-.75,4.5)--(3.3,4.5);
\draw[blue](3.5,4.5)--(3.75,4.5);
\draw[color=red, mid arrow=red](.5,3)--(1.5,3);
\draw[color=red,mid arrow=red](1.5,3)--(2.5,3);
\draw[color=orange,mid arrow=orange] (3,3.5) to[out=60,in=270](3.4,4.5) to[out=90,in=0](3,5) to[out=180,in=80](2.75,4.6);
\draw[color=orange](2.75, 4.4) to[out=275,in=120](3,3.5);
\filldraw (1.5,3) circle (2pt);
\node at (1.35,2.85){\scriptsize$*$};
\node at (3.45,4){\scriptsize$\alpha$};
\node at (1,3.3){\scriptsize$\gamma$};

\shade[ball color=cyan!70!blue, opacity=.3] (0,0) circle (.5cm);
\shade[ball color=cyan!70!blue, opacity=.3] (3,0) circle (.5cm);
\draw[blue] (-.75,1.5)--(1.7,1.5);
\draw[blue](1.9,1.5)--(3.75,1.5);
\draw[color=red, mid arrow=red](.5,0)--(1.5,0);
\draw[color=red,mid arrow=red](1.5,0) to[out=0,in=260](1.8,1.5) to[out=80,in=180](2.2,2) to[out=0,in=105](2.5,1.6);
\draw[color=red](2.55,1.4) to[out=280,in=150](2.5,0);
\filldraw (1.5,0) circle (2pt);
\node at (1.35,-.15){\scriptsize$*$};
\node at (1,.3){\scriptsize$\gamma\alpha$};
\end{tikzpicture}
%
\begin{tikzpicture}
\tikzset{
  mid arrow/.style={
    postaction={
      decorate,
      decoration={
        markings,
        mark=at position 0.52 with {\arrow[#1]{stealth}}
      }
    }
  },
  mid arrow/.default=black
}
\shade[ball color=cyan!70!blue, opacity=.3] (0,3) circle (.5cm);
\shade[ball color=cyan!70!blue, opacity=.3] (3,3) circle (.5cm);
\draw[blue] (-.75,4.5)--(.3,4.5);
\draw[blue](.5,4.5)--(3.75,4.5);
\draw[color=red, mid arrow=red](.5,3)--(1.5,3);
\draw[color=red,mid arrow=red](1.5,3)--(2.5,3);
\draw[color=orange,mid arrow=orange] (0,3.5) to[out=60,in=270](0.4,4.5) to[out=90,in=0](0,5) to[out=180,in=80](-.25,4.6);
\draw[color=orange](-.25, 4.4) to[out=275,in=120](0,3.5);
\filldraw (1.5,3) circle (2pt);
\node at (1.35,2.85){\scriptsize$*$};
\node at (2,3.3){\scriptsize$\gamma$};
\node at (.45,4){\scriptsize$\alpha$};

\shade[ball color=cyan!70!blue, opacity=.3] (0,0) circle (.5cm);
\shade[ball color=cyan!70!blue, opacity=.3] (3,0) circle (.5cm);
\draw[blue] (-.75,1.5)--(.7,1.5);
\draw[blue](.9,1.5)--(3.75,1.5);
\draw[color=red,mid arrow=red](1.5,0)--(2.5,0);
\draw[color=red,mid arrow=red](0.5,0) to[out=0,in=260](0.8,1.5) to[out=80,in=180](1.2,2) to[out=0,in=105](1.5,1.6);
\draw[color=red](1.55,1.4) to[out=280,in=150](1.5,0);

\filldraw (1.5,0) circle (2pt);
\node at (1.35,-.15){\scriptsize$*$};
\node at (2,0.3){\scriptsize$\alpha\gamma$};
\end{tikzpicture}
\end{subfigure}

\caption{A schematic picture of Whitehead moves of type (i) and type (ii). (The blue curve linking $\alpha$ is meant to indicate that $\alpha$ is non-trivial). In (a), the sphere cuts off some $M_i$  or an $S^1\times S^2$ factor of $M$, and $\gamma$ is a loop in that component. In (b), the spheres represent the two feet of a 1-handle attached to $S^3$ for the construction of $M$ as in Figure \ref{fig:prime-decomposition}, and $\gamma$ is a loop which goes once over the 1-handle. }
\label{fig:Whitehead}
\end{figure}

The image on the left in Figure \ref{fig:Whitehead} shows a homeomorphism that induces a Whitehead automorphism of type (i). The two images on the right show homeomorphisms that induce Whitehead automorphisms of type (ii). These figures can be viewed as local pictures in $\R^3$ that can be appropriately inserted into $M$. A factor inversion is induced by a homeomorphism that exchanges the attaching spheres of the two boundary components of $S^2\times[0,1]$ in a small ball in $S^3$.

\section{\boldmath Sphere systems and Grushko trees }\label{sec:hatcher-generalized}

In this section we generalize Hatcher's argument to prove the following theorem.

\begin{thm}\label{thm:sphere-tree}
Fix a closed, orientable $3$-manifold $M$. The map $\Phi:\mathcal S(M)\to\mathcal X(\pi_1M)$ (sending an essential, filling sphere system to its dual Grushko tree) is a bijection that's equivariant with respect to the map $\Mod(M)\to\Out(\pi_1M)$. 
\end{thm}

\begin{rmk}\label{rmk:sphere-twist}
Theorem \ref{thm:sphere-tree} implies that the action of $\Mod(M)$ on $\mathcal S(M)$ factors through $\Out(\pi_1M)$. Indeed, this follows immediately from the assertion that $\Phi$ is injective and equivariant. 
\end{rmk}

Given Theorem \ref{thm:sphere-tree}, the proof of Theorem \ref{thm:mod-sphere-decomposition} proceeds as in the special case of $M=(S^1\times S^2)^{\#n}$ given in the introduction. 

\begin{proof}[Proof of Theorem \ref{thm:mod-sphere-decomposition}]
Fix a finite subgroup $G<\Mod(M)$, and let $G'$ denote the image of $G$ in $\Out(\Gamma)$, where $\Gamma=\pi_1(M)$. By Hensel--Kielak (Theorem \ref{thm:hensel-kielak}) the action of $G'<\Out(\Gamma)$ on $\mathcal X(\Gamma)$ has a fixed point $T\in\mathcal X(\Gamma)$. By Theorem \ref{thm:sphere-tree} (see also Remark \ref{rmk:sphere-twist}), this implies that the action of $G$ on $\mathcal S(M)$ has a fixed point, as desired. 
\end{proof}

We now turn to the proof of Theorem \ref{thm:sphere-tree}. 
Defining an inverse directly, as in the proof of Proposition \ref{prop:hatcher}, is more subtle because the homomorphism $\Homeo(M)\to\Out(\pi_1M)$ is not generally surjective. Instead we will show directly that the map is injective and surjective. 

For the proof, we will use the same notation as the preceding sections. We write \[M=M_1\#\cdots\#M_m\#(S^1\times S^2)^{\#n}\>\>\>\text{ and }\>\>\>\Gamma=\Gamma_1*\cdots*\Gamma_m*F_n\] 
 for ``the" prime and Grushko decompositions of $M$ and $\Gamma=\pi_1(M)$, respectively (understanding, of course, that these decompositions are not unique in a strict sense).

\begin{proof}[Proof of Theorem \ref{thm:sphere-tree}]
The equivariance of $\Phi$ is obvious. Our main task is to show that $\Phi$ is a bijection. If $M$ is irreducible, then bijectivity is obvious because $\mathcal{S}(M)$ and $\mathcal X(\pi_1M)$ are both singletons ($\mathcal S(M)$ contains only the empty sphere system, and $\mathcal X(\pi_1M)$ contains only the tree with a single vertex). In the remainder of the proof we focus on the case when $M$ is reducible. 

\paragraph{\boldmath Injectivity of $\Phi$.} 

Fix $S,S'\in\mathcal S(M)$. Denote by $T,T'$ the trees dual to the spheres $\wtil S,\wtil S'$ in the universal cover $\wtil M$. Suppose that there is a $\Gamma$-equivariant isomorphism $\phi:T\to T'$. Since edges of $T,T'$ are dual to elements of $\wtil S,\wtil S'$, respectively, $\phi$ defines a bijection $\phi:\wtil S\to \wtil S'$. 

It suffices to show that $s$ and $\phi(s)$ are homologous for each component $s$ of $\wtil S$. Indeed, if $s,\phi(s)$ are homologous, then they are homotopic by Hurewicz's theorem $\pi_2(\wtil M)\cong H_2(\wtil M)$. Applying this to each component, we conclude that  the sphere systems $S,S'$ in $M$ are homotopic, hence isotopic by a theorem of Laudenbach \cite{laudenbach}; see also \cite[Thm.\ 2.2]{hatcher-wahl}. 

Fix a component $s$ of $\wtil S$, and set $s':=\phi(s)$, which is a component of $\wtil S'$. Without loss of generality, we assume $s,s'$ are in general position. To show $s,s'$ are homologous, let $C$ be the closure of one of the two components of $\wtil M\setminus s$ and let $C'$ be the closure of the corresponding component of $\wtil M\setminus s'$. Here corresponding means with respect to the isomorphism $\phi:T\to T'$, noting that $s$ is dual to an edge of $T$, and choosing $C$ is the same as choosing an endpoint of this edge. To show $s,s'$ are homologous it suffices to show that the symmetric difference $C\Delta C'$ is compact, for then we can use $C\Delta C'$ to build a 3-chain whose boundary is $[s]-[s']$. 

We show $C\Delta C'$ is compact. Fix a basepoint $x\in \wtil M\setminus (\wtil S\cup\wtil S')$. Since $M$ is compact, we can choose a compact fundamental domain $K\ni x$ for the action of $\Gamma$ on $\wtil M$. Since finitely many $\Gamma$ translates of $K$ intersect the compact set $s\cup s'$, to show $C\Delta C'$ is compact, it is enough to show that there are finitely many  $\gamma\in \Gamma$ so that $\gamma K$ is contained in the interior of $C\Delta C'$. We can translate this to a statement about the trees $T,T'$. Let $\bar C\subset T$ and $\bar C'\subset T'$ be the subtrees corresponding to $C,C'$, respectively. Let $v$ be the vertex of $T$ that's dual to the component of $M\setminus\wtil S$ containing the basepoint $x$, and define $v'\in T'$ similarly. Now it suffices to show that there are finitely many $\gamma\in\Gamma$ such that either (i) $\gamma v\in\bar C$ and $\gamma v'\notin\bar C'$ or (ii) $\gamma v\notin\bar C$ and $\gamma v'\in\bar C'$. Observe that if $\gamma v\in \bar C$ and $\gamma v'\notin\bar C'$, then $\phi(\gamma v)=\gamma\phi(v)\in\bar C'$ and $\gamma v'\notin\bar C'$, so the path $\gamma[\phi(v),v']$ passes through the edge of $T'$ dual to $s'$. Since the path $[\phi(v),v']$ is compact, only finitely many $\gamma$ translates have this property. This completes the proof that $C\Delta C'$ is compact.

\paragraph{\boldmath Surjectivity of $\Phi$.} Let $T$ be a Grushko tree for $\Gamma$. We want to find a sphere system $S\subset M$ whose dual Grushko tree is $T$.  It will be convenient to denote the action of $\Gamma$ on $T$ by $\rho:\Gamma\to\Aut(T)$. 

{\it Step 1.} First we construct a (essential, filling) sphere system $S'\subset M$ whose dual Grushko tree $T'$ differs from $T$ by the action of $\alpha\in\Aut(\Gamma)$. Concretely, $T'\cong T$ (equivariantly isomorphic) and if $\rho':\Gamma\to\Aut(T)$ is the corresponding action, then $\rho=\rho'\circ\alpha$. 

The quotient $X=T/\Gamma$ has the structure of a graph of groups $(X,\mathcal G)$. We use this to build a manifold $M_X$ homeomorphic to $M$ as follows. Each vertex group $\mathcal G_v$ is either trivial or isomorphic to one of the subgroups $\Gamma_i$. Let $V_{t}$ (resp.\ $V_{nt}$) be the vertices of $X$ with trivial (resp.\ nontrivial) vertex group. Choose an enumeration $V_{nt}=\{v_1,\ldots,v_m\}$ so that $\mathcal G_{v_i}\cong\pi_1(M_i)$. To form $M_X$, for each $v_i\in V_{nt}$ take a copy of $M_i$ and remove the interior of a collection of $\deg(v_i)$ disjoint 3-balls, and for each $v\in V_t$ do the same but with a copy of $S^3$; glue these compact manifolds along their 2-sphere boundaries according to the edge relations in $X$; the result is $M_X$. The 2-spheres where the gluing occurs define an essential, filling sphere system $S_X\subset M_X$. 

Choose a homeomorphism $M_X\cong M$ and let $S'\subset M$ be the image of the sphere system $S_X\subset M_X$ in $M$. Set $T':=\Phi(S')$. By construction $T,T'\in\mathcal X(\Gamma)$ give the same \emph{unmarked} graph of groups, so $T,T'$ are in the same $\Aut(\Gamma)$ orbit (Lemma \ref{lem:orbits}). More precisely, 
there is a $\Gamma$-equivariant isomorphism $T\cong T'$ and $\alpha\in\Aut(\Gamma)$ so that $\rho=\rho'\circ\alpha$, where $\rho,\rho':\Gamma\to\Aut(T)$ are the two given $\Gamma$-actions on $T$.

{\it Step 2.} Now we show $(T,\rho)$ is in the image of $\Phi$ given our knowledge that $(T,\rho')$ is in the image and $\rho=\rho'\circ\alpha$.  Recall from Section \ref{sec:image} that $\Aut(\Gamma)$ can be expressed as a product 
\[\Aut(\Gamma)=W\left[\prod\Aut(\Gamma_i)\rtimes P_\Gamma\right],\] and the subgroup $W$ is contained in the image of $\Homeo(M,*)\to\Aut(\Gamma)$. 
Thus we may write $\alpha=w\beta\sigma$, where $w\in W$ and $\beta\in \prod_{i=1}^m\Aut(\Gamma_i)$ and $\sigma\in P_\Gamma$.

Since $w$ is induced by a homeomorphism $f$, then setting $S''=f(S')$, we have 
\[\Phi(S'')=\Phi(f(S'))=(T,\rho'\circ f_*)=(T,\rho'\circ w).\] 
Therefore, it remains to consider the case $\alpha=\beta\sigma$ belongs to the subgroup generated by factor automorphisms and permutations.

If $\beta\in\prod_{i=1}^m\Aut(\Gamma_i)$ is a factor automorphism, then we argue that $(T,\rho')$ and $(T,\rho'\circ\beta)$ are equivariantly isomorphic, or equivalently, there exists $\phi\in\Aut(T)$ so that $\rho'\circ\beta=\phi\> \rho'\> \phi^{-1}$. To define $\phi$, first choose a fundamental domain for $\rho'(\Gamma)\curvearrowright T$, and use this to label vertices by cosets of vertex stabilizers in the fundamental domain. Then define $\phi$ on vertices by $\phi(v_{x\Lambda})=v_{\beta(x)\beta(\Lambda)}$. One checks that this preserves the edge relations, for example if $(x\Gamma_i,x\Gamma_j)$ is an edge, then $\phi$ sends this to $(\beta(x)\Gamma_i,\beta(x)\Gamma_j)$, which is also an edge. Another type of edge is $(x\Gamma_i,xt\Gamma_j)$, where $t$ is a stable letter (corresponding to the $F_n$ factor in $\Gamma$), and $\phi$ sends this to $(\beta(x)\Gamma_i,\beta(x)t\Gamma_j)$, which is again an edge. Thus $\phi$ defines an automorphism of $T$, and by construction $\rho'\circ\beta(\gamma)=\phi\circ\rho(\gamma)\circ\phi^{-1}$ for $\gamma\in\Gamma$ (check this on vertices of $T$).

It remains to consider the case $\alpha=\sigma\in P_\Gamma$ is a permutation automorphism. First observe that $P_\Gamma$, as a product of symmetric groups, is generated by transpositions, so it suffices to treat the case that $\alpha$ interchanges $\Gamma_j\leftrightarrow\Gamma_k$ for some fixed $j,k$ and is the identity on the other $\Gamma_i$. Our procedure for constructing $S$ in this case is as follows. First we build an auxiliary manifold $\what M$ via the same construction that we used to build $M_X\cong M$ but instead of the labeling $v_i\to M_i$ and $v_j\to M_j$ we use the labeling $v_i\to M_j$ and $v_j\to M_i$. This is illustrated in the following figure, where we focus on a subset of $M$ corresponding to a path in $X$ from $v_i$ to $v_j$ (in a fixed maximal tree of $X$, say).

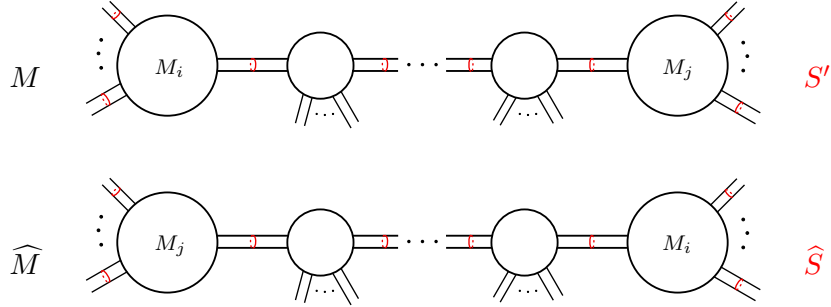
\begin{figure}[h]
\centering
\begin{overpic}[scale=.9]{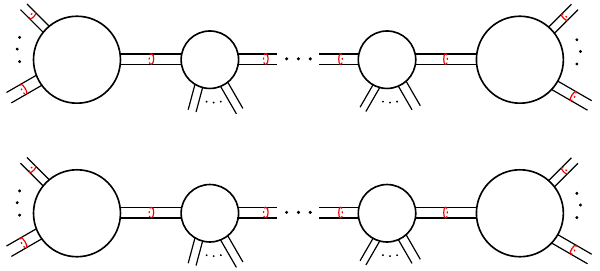}
\put (11,8.5){\scriptsize $M_j$}
\put (84.5,8.5){\scriptsize $M_i$}
\put (84.5,34){\scriptsize $M_j$}
\put (11,34){\scriptsize $M_i$}
\put (-10,5){$\widehat{M}$}
\put (-10,32){$M$}
\put (105,5){\textcolor{red}{$\widehat{S}$}}
\put (105, 32){\textcolor{red}{$S'$}}
\end{overpic}
\caption{The top image depicts the portion of $M$ with its sphere system $S'$  corresponding to the path in $X$ between the vertex $v_i$ (labeled $M_i$) and $v_j$ (labeled $M_j)$. The bottom shows the auxiliary manifold \smash{$\widehat{M}$} where $M_i$ and $M_j$ have been swapped, along with its natural sphere system \smash{$\widehat{S}$}.  }
\label{fig:wrong-label}
\end{figure}
The manifold $\what M$ comes with an obvious sphere system $\what S$, which is also pictured. To obtain $S$, consider the following alternative description of $\what M$ in Figure \ref{fig:auxiliary-plus} and perform a disk-dragging map that interchanges $M_j$ and $M_i$. This defines a homeomorphism \smash{$\what M\to M$}. Let $S$ be the image of \smash{$\what S$} under this homeomorphism.

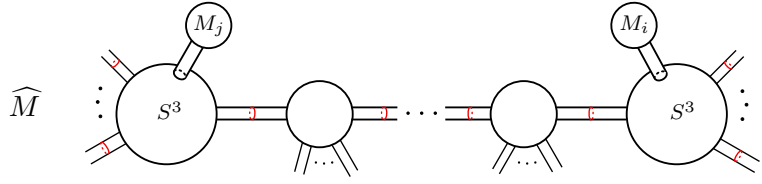
\begin{figure}[h]
\centering
\begin{overpic}[scale=.9]{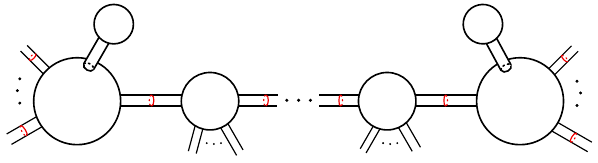}
\put (-10,10) {$\widehat{M}$}

\put (16.75,22.5) {\scriptsize$M_j$}
\put (78.7,22.5) {\scriptsize$M_i$}

\put (11.5,9.5) {\scriptsize$S^3$}
\put (86,9.5) {\scriptsize$S^3$}
\end{overpic}
\caption{The alternate description of the auxiliary manifold $\widehat{M}$, where $M_j$ and $M_i$ are isolated by 2-spheres separating them from a punctured $S^3$. Exchanging the disks bounded by these 2-spheres in each copy of $S^3$  defines a homeomorphism $\widehat{M}\rightarrow M$.  }
\label{fig:auxiliary-plus}
\end{figure}

The sphere system $S$ agrees with $S'$, except along the chosen path connecting $M_i$ to $M_j$. Figure \ref{fig:swap} below demonstrates the difference between $S'$ and $S$ in the case when $v_i$ and $v_j$ are connected by an edge. 

\begin{figure}[h]
\centering
\begin{overpic}[scale=.7]{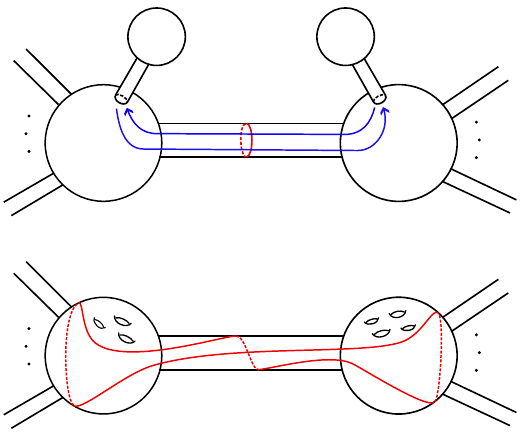}
\put (26.1,24){ \scriptsize$M_i$}
\put (59.3,24){ \scriptsize$M_j$}
\put (25.15,75.5){ \scriptsize$M_j$}
\put (61.3,75.5){ \scriptsize$M_i$}
\end{overpic}
\caption{ Before and after the disk dragging map that exchanges $M_j\leftrightarrow M_i$, when $v_i$ and $v_j$ are adjacent vertices in $X$.}
\label{fig:swap}
\end{figure}
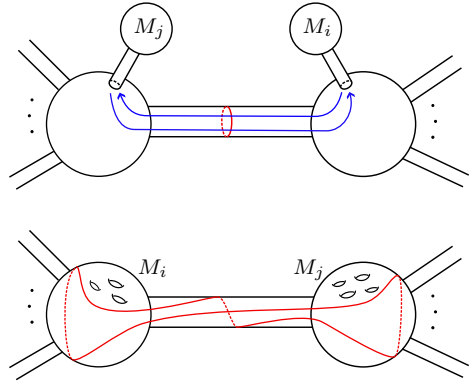

By direct computation, one checks that the Grushko tree dual to $S$ is isomorphic to $(T,\rho'\circ\alpha)$. For this, it helps to fix a maximal tree in $X$ and use this to present $\Gamma=\pi_1(M)$; we choose the path from $v_i$ to $v_j$ inside this maximal tree. In terms of the marking $\Gamma\to\pi_1(X,\mathcal G)$, the effect of the sequence of moves $S'\to\what S\to S$ exactly exchanges the vertices labeled by $\Gamma_i$ and $\Gamma_j$. 

This completes the proof of surjectivity and the proof of Theorem \ref{thm:sphere-tree}. 
\end{proof}

\section{\boldmath Image of $\Homeo(N)\to\Out(\pi_1N)$ for irreducible $N$ } \label{sec:image-irreducible}

Let $N$ be a closed, orientable irreducible 3-manifold.  In this section we compute the image of $\Homeo(N)\to\Out(\pi_1N)$. This will prove Theorem \ref{thm:image}, which complements Proposition \ref{prop:image-mod-out}.

First we explain that $\Homeo(N)\to\Out(\pi_1N)$ is surjective if $N$ is not covered by $S^3$.  In fact, $\Mod(N)\to\Out(\pi_1N)$ is an isomorphism, where our convention here is that $\Mod(N)=\pi_0\Homeo(N)$ (i.e.\ we include orientation-reversing homeomorphisms). This isomorphism follows from several well-known results. If $N$ is Haken, then Waldhausen \cite{waldhausen} showed that $\Mod(N)\to\Out(\pi_1N)$ is an isomorphism. If $N$ is not Haken, then it is either hyperbolic or Seifert fibered (by geometrization). In the hyperbolic case, surjectivity of $\Mod(N)\to\Out(\pi_1N)$ follows from Mostow rigidity, and injectivity holds by a theorem of Gabai--Meyerhoff--Thurston \cite{gabai-meyerhoff-thurston}. If $N$ is non-Haken Seifert fibered and also aspherical, then $\Mod(N)\to\Out(\pi_1N)$ is an isomorphism by McCullough--Soma \cite[Prop.\ 8.3]{mccullough-soma}. 

What remains is the spherical case. In this case, $\Mod(N)\to\Out(\pi_1N)$ is typically \emph{not} surjective. This can be explained using Reidemeister torsion and the $k$-invariant. 

For the remainder of this section, $N=S^3/\Gamma$  is a spherical manifold.  We will use the following facts. 
\begin{itemize}
\item An automorphism $\phi\in\Aut(\Gamma)$ is induced by a homotopy equivalence of $N$ if and only if $\phi$ acts on $H_3(\Gamma)\cong\Z/|\Gamma|\Z$ by multiplication by $\pm1$ \cite[Thm.\ pg.\ 91]{plotnick-k-invariant}. 
\item The homomorphism $\Mod(N)\to\Out(\Gamma)$ is injective, unless $N=S^3$ or $\R P^3$. Indeed, with these exclusions, any diffeomorphism that induces the identity on $\pi_1$ is homotopic to the identity; c.f.\ \cite[Lem.\ 3.3]{mccullough-isometries-elliptic}. This fails for $S^3$ and $\R P^3$ because they have orientation-reversing diffeomorphisms, while every automorphism of $\pi_1(N)$ is trivial. On the other hand, for $S^3$ and $\R P^3$, for orientation-preserving homeomorphisms we have  $\Mod^+(N)\cong\Out(\Gamma)=1$. 
\end{itemize}

\begin{rmk}\label{rmk:k-invariants}
The condition that $\phi\in\Aut(\Gamma)$ acts on $H_3(\Gamma)$ by $\pm1$ can be interpreted in terms of $k$-invariants. 
A group $\Gamma$ acting freely on $S^{2k-1}$ has periodic cohomology of period $2k$; c.f.\ \cite[Ch.\ VI.9]{brown}. In particular $H_{2k-1}(\Gamma)\cong H^{2k}(\Gamma)\cong\Z/|\Gamma|\Z$. This group is generated by the Euler class $e\in H^{2k}(\Gamma)$ of the fibration $S^{2k-1}\to \frac{S^{2k-1}\times E\Gamma}{\Gamma}\to B\Gamma$. Indeed, the total space is homotopy equivalent to $N:=S^{2k-1}/\Gamma$ since $\Gamma$ acts freely on $S^{2k-1}$, so in the spectral sequence of this fibration, the transgression $H^{2k-1}(S^{2k-1})\to H^{2k}(B\Gamma)$ is surjective because $H^{2k}(N)=0$. This Euler class also coincides with the $(2k-1)$-st $k$-invariant of $N$. In the Postnikov tower $X_n$ of $N$, the space $X_{2k-1}$ is a fibration $K(\Z,2k-1)\to X_{2k-1} \to K(\Gamma,1)$, and the associated $k$-invariant $k\in H^{2k}(B\Gamma;\Z)$ is the Euler class of this fibration. By the Postnikov construction, there is a map $N\to X_3$ that extends to a map between the fibrations $S^{2k-1}\to N\to B\Gamma$ and $K(\Z,2k-1)\to X_{2k-1} \to K(\Gamma,1)$. By naturality of the Euler class, $k=e$. Thus $\phi$ acts by $\pm1$ on $H_3(\Gamma)$ if and only if it preserves the $k$-invariant of $N$ (up to sign, since the $k$-invariant is only well-defined up to a sign).
\end{rmk}

\subsection{Lens spaces}

Fix $(m,q)$ relatively prime. Recall that $L(m,q)$ denotes the quotient of $S^3\subset\C^2$ by the $\Z/m\Z$ action generated by the homeomorphism $(z_0,z_1)\mapsto(\zeta z_0,\zeta^q z_1)$, where $\zeta=e^{2\pi i/m}$. 
For example, $L(1,q)\cong S^3$ and $L(2,1)\cong\R P^3$. 

\begin{prop}\label{prop:lens-mod} Fix $(m,q)$ relatively prime with $m\ge1$.
\begin{enumerate}
\item[(a)] If $m\in\{1,2\}$, then $\Mod(L(m,q))\cong\Z/2\Z$ acts trivially on $\pi_1(L(m,q))\cong\Z/m\Z$.
\end{enumerate}
For the remaining cases, assume $m\ge3$.
\begin{enumerate}
\item[(b)] If either $q=\pm1\pmod m$ or $q^2\neq\pm1\pmod m$, then $\Mod(L(m,q))\cong\Z/2\Z$ acts on $\pi_1(L(m,q))\cong\Z/m\Z$ by multiplication by $\pm1$. 
\end{enumerate}
For the remaining cases, assume $q\neq\pm1\pmod m$. 
\begin{enumerate}
\item[(c)] If $q^2=1\pmod m$ then $\Mod(L(m,q))\cong \Z/2\Z\times\Z/2\Z$ acts on $\pi_1(L(m,q))$ by multiplication by $\{\pm 1,\pm q\}\cong\Z/2\Z\times\Z/2\Z$. 
\item[(d)] If $q^2=-1\pmod m$ then $\Mod(L(m,q))\cong \Z/4\Z$ acts on $\pi_1(L(m,q))$ by multiplication by $\{\pm 1,\pm q\}\cong\Z/4\Z$. 
\end{enumerate}
\end{prop}

\begin{proof}
The mapping class computations are contained in McCullough \cite[Table 3]{mccullough-isometries-elliptic}. Given the mapping class group computation, (a) is obvious since the fundamental group is either trivial or $\Z/2\Z$ (and both have trivial automorphism group). 

In the remaining cases (b)--(d), the homomorphism $\Mod(L(m,q))\to\Out(\Z/m\Z)$ is injective. To determine the image, we use the Reidemeister torsion, whose value is
\[\tau(L(m,q))=(\zeta-1)(\zeta^r-1),\]
where $rq=1\pmod m$ and this expression is only well defined modulo multiplication by an element of $\{\pm\zeta^k\}$ (see e.g.\ \cite{ranicki}). Recall that the torsion can be used to show that if $L(m,q)$ is homeomorphic to $L(m,q')$, then $q'=\pm q^{-1}\pmod m$. The same analysis also shows that if $a\in(\Z/m\Z)^\times$, then the automorphism of $\Z/m\Z$ given by multiplication by $a$ is induced by a homeomorphism of $L(m,q)$ only if the subsets $\{a,-a,ar,-ar\}$ and $\{1,-1,r,-r\}$ of $\Z/m\Z$ agree. To see this, first note that the torsion is a homeomorphism invariant, so if multiplication by $a$ (which we write as $\zeta\mapsto\zeta^a$) is induced by a homeomorphism, then we must have an equality
\[(\zeta-1)(\zeta^r-1)=\pm\zeta^k(\zeta^a-1)(\zeta^{ar}-1)\]
for some $k$. Multiplying each side by its conjugate gives an equation
\[(\zeta-1)(\zeta^{-1}-1)(\zeta^r-1)(\zeta^{-r}-1)=(\zeta^a-1)(\zeta^{-a}-1)(\zeta^{ar}-1)(\zeta^{-ar}-1),\]
and a lemma of Franz implies that the sets of exponents must agree. Compare with \cite[pg. 98]{cohen}. 

In each case (b)--(d), we compute that the possible values for $a$ are either $\{\pm1\}$ or $\{\pm1,\pm q\}$, as in the statement of the proposition. This only gives an upper bound on the image of $\Mod(L(m,q))\to\Out(\Z/m\Z)$, but in each case this upper bound is a group isomorphic to $\Mod(L(m,q))$, and since $\Mod$ injects into $\Out$, this allows us to conclude. 
\end{proof}

\subsection{Non-lens-space spherical manifolds}

We divide into seven cases for $\Gamma$, following  McCullough \cite{mccullough-isometries-elliptic}, where the mapping class groups of these manifolds are computed. 
\begin{enumerate}[I.]
\item $Q_8$ or $Q_8\times C_n$ with $(2,n)=1$
\item $D_{4m}^*$ or $D_{4m}^*\times C_n$ with $m>2$ and $(2m,n)=1$
\item $T_{24}^*$ or $T_{24}^*\times C_n$ with $(6,n)=1$
\item $O_{48}^*$ or $O_{48}*\times C_n$ with $(6,n)=1$
\item $I_{120}^*$ or $I_{120}^*\times C_n$ with $(30,n)=1$
\item Index 2 diagonal in $D_{4m}^*\times C_{4n}$, quotient by the diagonal central subgroup, where $(2m,n)=1$ and $(m,2)=1$
\item Index 3 diagonal in $T_{24}^*\times C_{6n}$, quotient by the diagonal central subgroup, where $3\mid n$ and $(n,2)=1$ 
\end{enumerate}

\paragraph{Case I (quaternion).} 

If $\Gamma=Q_8$, then $\Out(\Gamma)\cong S_3$. On the other hand, McCullough computes $\Mod(N)\cong S_3$. Thus $\Mod(N)\to\Out(\Gamma)$ is an isomorphism (since the map is injective).

If $\Gamma=Q_8\times C_n$ with $n>1$ odd, then $\Out(Q_8\times C_n)\cong S_3\times(\Z/n\Z)^\times$. To restrict the image of $\Mod(N)\to\Out(\Gamma)$, observe that if an automorphism $\alpha$ of $\Gamma=Q_8\times C_n$ is induced by a homeomorphism $f$ of $N$, then since it preserves the subgroup $\{1\}\times C_n$, it lifts to an automorphism of $S^3/C_n=L(n,1)$. Thus if $\alpha=f_*$, then $\alpha$ acts on the $C_n$ factor by multiplication by $\pm1$ by Proposition \ref{prop:lens-mod}. On the other hand, McCullough computes $\Mod(N)\cong  S_3\times C_2$, so the preceding constraint determines the image of $\Mod(N)\to\Out(\Gamma)$ exactly. 

\begin{rmk}[Roadmap]
Through this example we see what will be the general pattern of our computation:
\begin{enumerate}
\item First assume $\Gamma$ does not have the cyclic factor. Compute $\Out(\Gamma)$. Use the $k$-invariant and lens space covers to constrain the image of $\Mod(N)\to\Out(\Gamma)$. Compare with McCullough's computation of $\Mod(N)$ and deduce there is only one possible injection $\Mod(N)\to\Out(\Gamma)$ subject to the constraint.
\item When $\Gamma=\Gamma_0\times C$ has a cyclic factor, since the orders of $C$ and $\Gamma_0$ are relatively prime $\Aut(\Gamma)\cong\Aut(\Gamma_0)\times\Aut(C)$. Now use the fact that $N$ has a lens space cover to constrain the image of $\Mod(N)\to\Aut(\Gamma)\to\Aut(C)$. Compare again with McCullough's computation and conclude. 
\end{enumerate} 
\end{rmk}

\paragraph{Case II (binary dihedral).}

The group $\Gamma=D_{4m}^*$ is also known as the binary dihedral group and also as a dicyclic group. In other words, it is both an extension $1\to C_2\to D_{4m}^* \to D_{2m}\to1$ and an extension $1\to C_{2m}\to D_{4m}^*\to C_2\to1$. To compute $\Aut(\Gamma)$, we use the latter perspective. There is a presentation 
\[D_{4m}^*=\langle a,x\mid a^{2m}=1, x^{-1}ax=a^{-1}, x^2=a^m\rangle.\]
Here $a$ generates the kernel of the map $D_{4m}^*\to C_2$. 

We show $\Aut(\Gamma)\cong C_{2m}\rtimes\Aut(C_{2m})$. Since we assume $m>2$, the subgroup $C_{2m}<D_{4m}^*$ is characteristic (it is the unique cyclic subgroup of order $2m$), so there is a homomorphism $\Aut(\Gamma)\to\Aut(C_{2m})$ and this is obviously a split surjection (consider automorphisms of $\Gamma$ of the form $a\mapsto a^k$ and $x\mapsto x$ for $k$ relatively prime to $2m$). To analyze the kernel, note (using the presentation) that elements of $\Gamma$ can be written in a normal form $x^ia^j$ where $i=0,1$ and $j=0,\ldots,2m-1$. Then the kernel consists of automorphisms of the form $a\mapsto a$ and $x\mapsto xa^j$. These form a cyclic subgroup generated by the automorphism $a\mapsto a$ and $x\mapsto xa$. 

We compute $\Out(\Gamma)\cong C_2\times \frac{\Aut(C_{2m})}{\{\pm1\}}$, since, with respect to the isomorphism $\Aut(\Gamma)\cong C_{2m}\rtimes\Aut(C_{2m})$, the inner automorphism group is identified with $C_m\rtimes\{\pm1\}$. Indeed, conjugation by $a$ acts by $a\mapsto a$ and $x\mapsto xa^{-2}$ and conjugation by $x$ acts by $a\mapsto a^{-1}$ and $x\mapsto x$. The factor $C_2<\Out(\Gamma)$ is generated by the automorphism $a\mapsto a$ and $x\mapsto xa$. Projection of $\phi\in\Out(\Gamma)$ to the factor $\Aut(C_{2m})$ is given by the action on $\langle a\rangle=C_{2m}$.

For the topological constraint, if $\phi\in\Aut(\Gamma)$ is induced by a homotopy equivalence, then it must act by $\pm1$ on  $H^4(\Gamma;\Z)\cong\Z/|\Gamma|\Z$. Consequently, the action of $\phi$ on $H^4(C_{2m};\Z)\cong\Z/2m\Z$ (by restriction) must also be by multiplication by $\pm1$. From this and Proposition \ref{prop:lens-mod} we see no element of $\Out(\Gamma)\cong C_2\times \frac{\Aut(C_{2m})}{\{\pm1\}}$ that projects nontrivially to the second factor is in the image of $\Mod(N)$. Thus the image of $\Mod(N)\to\Out(\Gamma)$ is contained in $C_2<\Out(\Gamma)$.  

On the other hand, McCullough proves $\Mod(N)\cong C_2$, so the image of $\Mod(N)\to\Out(\Gamma)$ is the subgroup generated by the automorphism sending $a\mapsto a$ and $x\mapsto xa$. 

If $\Gamma=D_{4m}^*\times C_{n}$, then using the lens space trick we see that the image of $\Mod(N)$ is contained in $C_2\times C_2\subset \Aut(D_{4m}^*)\times\Aut(C_n)$, and McCullough computes $\Mod(N)=C_2\times C_2$, so this computes the image precisely. 

\paragraph{Case III (binary tetrahedral).} For the binary tetrahedral group $\Gamma=T_{24}^*$, there is an extension 
\[1\to C_2\to T_{24}^*\to A_4\to1.\]

One computes that $\Aut(T_{24}^*)\cong\Aut(A_4)$ and $\Out(T_{24}^*)\cong\Out(A_4)\cong\Z/2\Z$. This is well-known. One route to this is to 
make use of the fact that the extension defining $T_{24}^*$ is isomorphic to
\[1\to \{\pm I\}\to\SL_2(\mathbb F_3)\to\PSL_2(\mathbb F_3)\to1.\]

Now McCullough computes $\Mod(N)\cong C_2$, so $\Mod(N)\to\Out(\Gamma)$ is an isomorphism. 

If $\Gamma=T_{24}^*\times C_n$, then we find that the image of $\Mod(N)\to\Out(\Gamma)$ is contained in $C_2\times C_2<\Out(T_{24}^*)\times\Out(C_n)$. Since McCullough computes $\Mod(N)\cong C_2\times C_2$, this gives the image precisely. 

\paragraph{Case IV (binary octahedral).} 

Set $\Gamma=O_{48}^*$, the binary octahedral group. In this case McCullough computes $\Mod(N)=1$, so $\Mod(N)\to\Out(\Gamma)$ is also trivial. 

It is interesting to elaborate on this because $\Out(\Gamma)\cong C_2$ is not trivial; see e.g.\ Plotnick \cite[\S3]{plotnick-k-invariant}. Plotnick also computes that the generator $\phi\in \Out(\Gamma)$ acts on $H_3(\Gamma)\cong\Z/48\Z$ by multiplication by $25$. In particular, this implies that $\phi$ is not induced by a homotopy equivalence, which gives another explanation for why the image of $\Mod(N)\to\Out(\Gamma)$ is trivial (without the need for McCullough's computation). 

As in the preceding analysis, if $\Gamma=O_{48}^*\times C_n$, then the image of $\Mod(N)\to\Out(\Gamma)$ is contained in $\{1\}\times C_2<\Out(O_{48}^*)\times\Out(C_n)$. Since McCullough computes $\Mod(N)\cong C_2$, this gives the image precisely. 

\paragraph{Case V (binary icosahedral).} Set $\Gamma=I_{120}^*$, the binary icosahedral group. Here McCullough computes $\Mod(N)=1$ so again $\Mod(N)\to\Out(\Gamma)$ is trivial. This case is also discussed by Plotnick. Here $\Out(\Gamma)\cong C_2$ and a generator acts on $H_3(\Gamma)\cong\Z/120\Z$ by multiplication by 49, and thus is not induced by any homotopy equivalence. 

For $\Gamma=I_{120}^*\times C_n$, the image of $\Mod(N)\to\Out(\Gamma)$ is contained in $\{1\}\times C_2<\Out(I_{120}^*)\times\Out(C_n)$. Since McCullough computes $\Mod(N)\cong C_2$, this gives the image precisely. 

\paragraph{Case VI (index-2 diagonal).} 

Fix $m$ odd and $n=2^k$ for $k\ge1$. Define $\Gamma^*<C_{4n}\times D_{4m}^*$ as the pullback 
\[\begin{xy}
(-20,0)*+{\Gamma^*}="A";
(0,0)*+{D_{4m}^*}="B";
(-20,-15)*+{C_{4n}}="C";
(0,-15)*+{C_2}="D";
{\ar"A";"B"}?*!/_3mm/{};
{\ar@{->>}"C";"D"}?*!/^3mm/{\pi_1};
{\ar "A";"C"}?*!/_3mm/{};
{\ar@{->>} "B";"D"}?*!/_3mm/{\pi_2};
\end{xy}\]
with respect to maps $\pi_1,\pi_2$ as in the diagram. These maps are unique (the abelianization of $D_{4m}^*$ is $C_4$ when $m$ is odd). The group $\Gamma$ is defined as the quotient of $\Gamma^*$ by the diagonal central subgroup $C_2<\Gamma^*<C_{4n}\times D_{4m}^*$. (Here $\Gamma^*$ is naturally a subgroup of $S^3\times S^3$ and $\Gamma$ is its image in $\SO(4)$.) 

\begin{lem}\label{lem:index-2-iso-type}
The group $\Gamma$ defined above is isomorphic to $C_{m}\rtimes C_{4n}$, where the action $C_{4n}\to\Aut(C_m)=(\Z/m\Z)^\times$ has image $\{\pm1\}$. 
\end{lem}

Compare with Milnor's list of groups acting freely on $S^3$ \cite[Thm.\ 2]{milnor}. In determining the image of $\Mod(N)\to\Aut(\Gamma)$ it will be helpful to understand how $\Gamma$ embeds in $\SO(4)$. We see this from the isomorphism between $C_n\rtimes C_m$ and $\Gamma$ given in the proof of Lemma \ref{lem:index-2-iso-type}.

\begin{proof}[Proof of Lemma \ref{lem:index-2-iso-type}]
We work with the presentation 
\[C_{4n}\times D_{4m}^*=\langle t\mid t^{4n}=1\rangle\times \langle x,y\mid x^2=y^m=(xy)^2\rangle.\]
In the presentation of $D_{4m}^*$, the element $x^2=y^m$ is central. The relation $(xy)^2=x^2$ implies that $xyx^{-1}=y^{-1}$. This relation combined with the relation $x^2=y^m$ implies $x^4=y^{2m}=1$, since $y^m=x^2=x(x^2)x^{-1}=x(y^m)x^{-1}=y^{-m}$. 
The homomorphism $\pi_2:D_{4m}^*\to C_2$ is defined by $x\mapsto 1$ and $y\mapsto0$. Accordingly, $\Gamma^*=\{t^ix^jy^k: i+j\equiv0\>(2)\}$ is generated by $tx$ and $y$ (note that $t^2=(tx)^2y^m$). The group $\Gamma$ is obtained from $\Gamma^*$ by adding the relation $t^{2n}=y^m$. We find that $\Gamma$ is generated by $a=y^2$ and $b=tx$ (if $m=2\ell+1$, then $y=(tx)^{2n}(y^2)^{-\ell}$), and we have the following relations $a^m=1$, $b^{4n}=1$, and $bab^{-1}=a^{-1}$. This shows there is a surjection $C_m\rtimes C_{4n}\to \Gamma$, but since both groups have order $4nm$, this map is an isomorphism.
\end{proof}

\begin{rmk}
The case $n=1$ is not of interest. In this case $\Gamma$ is generated by $tx,y$ and has relations $(tx)^2=1$, $y^{2m}=1$, and $(tx)y(tx)^{-1}=y^{-1}$. This shows there is a surjective map $D_{4m}\to \Gamma$, which must be injective since $|\Gamma|=4m$. But $D_{4m}$ does not act freely on $S^3$. 
\end{rmk}

\begin{lem}\label{lem:index-2-aut}
Fix $n=2^k$ and $m$ odd, and set $\Gamma=C_m\rtimes C_{4n}$. Then $\Aut(\Gamma)$ is isomorphic to $C_m\rtimes\Aut(C_m)\times \Aut(C_{4n})$, and $\Out(\Gamma)$ is isomorphic to $\frac{(\Z/m\Z)^{\times}}{\{\pm1\}}\times(\Z/4n\Z)^\times$. 
\end{lem}

\begin{proof}
The subgroup $C_m<\Gamma$ is characteristic, so an automorphism of $\Gamma$ restricts to an automorphism of $C_m$ and induces an automorphism of $C_{4n}$. This defines a homomorphism $\Phi:\Aut(\Gamma)\to\Aut(C_m)\times\Aut(C_{4n})$, which is clearly a split surjection. Writing $C_m\rtimes C_{4n}=\langle a,b\rangle$, the kernel of $\Phi$ is cyclic of order $m$ generated by the automorphism $a\mapsto a$, $b\mapsto ba$. Thus $\Aut(\Gamma)\cong C_m\rtimes[\Aut(C_m)\times\Aut(C_{4n})]$. 

For the outer automorphism group, conjugation by $a$ acts by $a\mapsto a$ and $b\mapsto ba^{-2}$. Since $m$ is odd, this generates the kernel of $\Phi$. Conjugation by $b$ acts by $a\mapsto a^{-1}$ and $b\mapsto b$, which generates the subgroup $\{\pm1\}<(\Z/m\Z)^\times=\Aut(C_m)$. Thus $\Out(\Gamma)\cong \frac{(\Z/m\Z)^{\times}}{\{\pm1\}}\times(\Z/4n\Z)^\times$. 
\end{proof}

We now consider which elements of $\Out(\Gamma)$ are induced by a homeomorphism. Here we make use of the lens space trick. Every $\phi\in\Out(\Gamma)$ is represented by an automorphism that preserves the subgroups $C_m$ and $C_{4n}$ in $\Gamma$. The lens space cover of $N=S^3/\Gamma$ corresponding to $C_m<\Gamma$ is $L(m,1)$ and the lens space cover corresponding to $C_{4n}<\Gamma$ is $L(4n, 2n-1)$. To see this latter claim, recall that the subgroup $C_{4n}<\Gamma<\SO(4)$ lifts to $S^3\times S^3$ where it is generated by $(t,x)\in S^3\times S^3$. Here $t$ has order $4n$ and $x$ has order 4. The action on a unit quaternion $q\in S^3$ is given by $q\mapsto tqx^{-1}$. In coordinates $q=(q_1,q_2)\subset\C^2$ the action is given by 
\[(q_1,q_2)\mapsto (\zeta^{1+n}q_1,\zeta^{1-n}q_2),\]
where $\zeta=e^{2\pi i/4n}$. If $n=2^k\ge2$, then $(1+n)(1-2n)\equiv1-n\pmod{4n}$, so 
\[\zeta^{1-n}=(\zeta^{1+n})^{(1-2n)}.\]
This shows the quotient of $S^3$ by the subgroup $C_{4n}<\Gamma$ is isomorphic to $L(4n,1-2n)\cong L(4n,2n-1)$. Since $q=(2n-1)^2\equiv1\pmod{4n}$, by Proposition \ref{prop:lens-mod}, we see that the image of $\Mod(L(4n,2n-1))\to\Aut(\Z/4n\Z)$ is isomorphic to $C_2\times C_2$ generated by multiplication by $\{\pm1,\pm q\}$. 

From this we deduce that if $\phi\in\Aut(C_m)\times\Aut(C_{4n})$ is induced by a homeomorphism, then the projection of $\phi$ to $\frac{(\Z/m\Z)^\times}{\{\pm1\}}$ is trivial and the projection to $(\Z/4n\Z)^\times$ is contained in the subgroup $\{\pm 1,\pm 2n-1\}$. This gives a $C_2\times C_2$ upper bound on the image of $\Mod(N)\to\Out(\Gamma)$. Since McCullough computes $\Mod(N)\cong C_2\times C_2$, this computes the image precisely. 

We also need to consider the case $\Gamma=\big[C_m\rtimes C_{4n}\big]\times C_\ell$, where $\ell$ is odd and relatively prime to $m$. The analysis in this case is the same. The obstruction comes from the lens space $L(4n\ell, 2n\ell-1)$, which (together with $L(m,1)$ as before) constrains the image to $C_2\times C_2<\Aut(\Z/4n\ell\Z)$, the subgroup given by multiplication by $\pm1$ or $\pm (2n\ell-1)$. 

\paragraph{Case VII (index-3 diagonal).} Our analysis in this case is similar to the preceding example. We explain it briefly. Fix $n=3^k$ with $k\ge1$.  Define $\Gamma^*<C_{2n}\times T_{24}^*$ as the pullback 
\[\begin{xy}
(-20,0)*+{\Gamma^*}="A";
(0,0)*+{T_{24}^*}="B";
(-20,-15)*+{C_{2n}}="C";
(0,-15)*+{C_3}="D";
{\ar"A";"B"}?*!/_3mm/{};
{\ar@{->>}"C";"D"}?*!/^3mm/{\pi_1};
{\ar "A";"C"}?*!/_3mm/{};
{\ar@{->>} "B";"D"}?*!/_3mm/{\pi_2};
\end{xy}\]
with respect to maps $\pi_1,\pi_2$ as in the diagram. Define $\Gamma$ as the quotient of $\Gamma^*$ by the diagonal central subgroup $C_2<\Gamma^*<C_{2n}\times T_{24}^*$. (Aside: there are multiple choices for the maps $\pi_1,\pi_2$, since we can compose with an automorphism of $C_3$, but it turns out that they give isomorphic $\Gamma$, which are moreover conjugate in $\SO(4)$, so $N=S^3/\Gamma$ is independent of the choices.) 

\begin{lem}\label{lem:index-3-iso-type}
Fix $n=3^k$ with $k\ge1$. The group $\Gamma$ defined above is isomorphic to $Q_8\rtimes C_{n}$, where the action $C_n\to\Aut(Q_8)=S_3$ factors through $C_n\onto C_3<S_3$. 
\end{lem}

\begin{proof}
We work with the presentation 
\[C_{2n}\times T_{24}^*=C_{2n}\times[Q_8\rtimes C_3]=\langle t\mid t^{2n}=1\rangle\times \big[\langle x,y\mid x^2=y^2=(xy)^2\rangle\rtimes \langle s\mid s^3\rangle\big].\]
Here $sxs^{-1}=y$ and $sys^{-1}=xy$. We define $\pi_1$ by $t\mapsto 1$ and $\pi_2$ by $s\mapsto 1$ and $x,y\mapsto0$. 

The group $\Gamma^*=\{t^iqs^j: q\in Q_8, i+j\equiv0\>(3)\}$ is generated by $x,y$ and $ts^2$. The quotient $\Gamma$ can be obtained by adding the relation $(ts^2)^n=x^2$. Accordingly we find $\Gamma\cong Q_8\rtimes C_{n}$, where $C_n$ is generated by the image of $(ts^2)^2=t^2s$.   
\end{proof}

\begin{rmk}
In the case $k=1$ of the lemma, we recover $\Gamma=T_{24}^*$. 
\end{rmk}

\begin{lem}\label{lem:index-3-aut}
Fix $n=3^k$, $k\ge1$ and $m$ odd, and set $\Gamma=Q_8\rtimes C_{n}$. Then $\Aut(\Gamma)$ is $C_{3^{k-1}}\times S_4$, and $\Out(\Gamma)$ is isomorphic to $C_{3^{k-1}}\times C_2\cong C_{2\cdot 3^{k-1}}$. 
\end{lem}

\begin{proof}

Fix a generating set $\Gamma=Q_8\rtimes C_n=\langle x,y\rangle\rtimes \langle u\rangle$ with $uxu^{-1}=y$ and $uyu^{-1}=xy$. The subgroup $Q_8<\Gamma$ is characteristic, so there is a homomorphism 
\[\Phi:\Aut(\Gamma)\to\Aut(Q_8)\cong S_4.\]

The map $\Phi$ is surjective, i.e.\ each automorphism of $Q_8$ extends to an automorphism of $\Gamma$. Indeed, one checks that for each $\sigma\in\Aut(Q_8)$, there exists a unique $\epsilon\in\{\pm1\}$ and $q\in Q_8$ such that $x\mapsto\sigma(x)$, $y\mapsto\sigma(y)$, $u\mapsto u^\epsilon q$ defines an automorphism of $\Gamma$. 

The kernel of $\Phi$ consists of automorphisms of the form $x\mapsto x$, $y\mapsto y$, and $u\mapsto u^i$. This is an automorphism if and only if $i\equiv 1\pmod 3$. In other words, $\ker(\Phi)$ is isomorphic to the kernel of the restriction map $\Aut(C_n)\to \Aut(C_3)$; this group is isomorphic to $C_{3^{k-1}}$ (since $(\Z/3^k\Z)^\times$ is cyclic of order $3^{k-1}$). Since an automorphism of $\Gamma$ induces an automorphism of $\Gamma/Q_8=C_n$, we see that $\ker(\Phi)$ is a retract of $\Gamma$, so in fact the extension $1\to C_{3^{k-1}}\to \Aut(\Gamma)\to S_4\to1$ splits, as desired. 

To compute $\Out(\Gamma)$, observe that the center of $\Gamma$ is isomorphic to $C_2\times C_{3^{k-1}}<Q_8\rtimes C_{3^k}$. Thus the inner automorphism group has order $12$. Since $\Aut(\Gamma)$ has order $8\cdot 3^k$, we can conclude from this that $\Out(\Gamma)\cong C_{3^{k-1}}\times C_2$. A generator of $C_2<\Out(\Gamma)$ is represented by the automorphism $x\mapsto y$, $y\mapsto x$, $u\mapsto u^{-1}y$, and if $2\ell\equiv1\pmod{3^k}$, then this automorphism has order 2. 
\end{proof}

Now we examine which elements of $\Out(\Gamma)$ are induced by homeomorphisms of $N=S^3/\Gamma$. For the remainder of this case, we assume $k\ge2$ (when $k=1$, the group $\Gamma$ is isomorphic to $Q_8\rtimes C_3\cong T_{24}^*$, which is covered by Case III above). First we focus on an element $\phi$ in the subgroup $C_{3^{k-1}}<\Out(\Gamma)$. Any such automorphism preserves $C_{3^k}<\Gamma$, so if $\phi$ is induced by a homeomorphism, then it lifts to an automorphism of the lens space $S^3/C_{3^k}\cong L(3^k,2\cdot3^{k-1}-1)$ (by a computation similar to the index-2 diagonal case). For $q=2\cdot3^{k-1}-1$ and $k\ge2$, we have $q^2\neq\pm1\pmod{3^k}$, so by Proposition \ref{prop:lens-mod}, this implies that $\phi=1$ (since elements of $C_{3^{k-1}}$ act on $C_{3^k}$ by multiplication by units congruent to $1\pmod 3$, as observed in the proof of Lemma \ref{lem:index-3-aut}). 

From this constraint and the group theory of $\Out(\Gamma)=C_{3^{k-1}}\times C_2$, we deduce that if $\phi\in\Out(\Gamma)$ is induced by a homeomorphism, then $\phi\in C_2$. On the other hand, McCullough proves that $\Mod(N)\cong C_2$, so the image of $\Mod(N)\to\Out(\Gamma)$ is $C_2$ and is generated by the automorphism $x\mapsto y$, $y\mapsto x$, $u\mapsto u^{-1}y$. 

The case when $n=3^k\ell$ where $\ell$ is relatively prime to $6$ is treated similarly. Again $\Mod(N)\cong C_2$, and if $\phi\in\Out(\Gamma)=\Aut(C_\ell)\times C_{3^{k-1}}\times C_2<\Aut(C_n)\times C_2$ is induced by a homeomorphism, then its action on $C_n$ is by $\pm1$.

\bibliographystyle{amsalpha}
\bibliography{references}

\end{document}